\documentclass[twoside,reqno,10pt]{amsart}
\usepackage{csquotes}
\usepackage{amsfonts, amscd}
\usepackage{graphicx}
\usepackage{amsmath,latexsym,amssymb,amsthm}
\usepackage{hyperref}
\usepackage{cite}
\usepackage{fancyhdr}
\usepackage{a4}
\usepackage{tikz-cd}
\usepackage{tikz}
\usepackage{mathtools}
\usepackage{mathabx}
\usepackage{bm}
\usepackage{amscd} 
\usepackage{mathrsfs}
\usepackage{relsize}
\usepackage{amsrefs}
\usepackage{comment}
\usepackage[all,cmtip]{xy}
\usepackage{fullpage}
\usepackage{pifont}
\usepackage{bbold}
\usepackage{comment}
\usepackage{tipa}
\usepackage{textcomp}
\usepackage{algpseudocode,algorithm}

\theoremstyle{plain}

\makeatletter
\newcommand*{\da@rightarrow}{\mathchar"0\hexnumber@\symAMSa 4B }
\newcommand*{\da@leftarrow}{\mathchar"0\hexnumber@\symAMSa 4C }
\newcommand*{\xdashrightarrow}[2][]{%
	\mathrel{%
		\mathpalette{\da@xarrow{#1}{#2}{}\da@rightarrow{\,}{}}{}%
	}%
}
\newcommand{\xdashleftarrow}[2][]{%
	\mathrel{%
		\mathpalette{\da@xarrow{#1}{#2}\da@leftarrow{}{}{\,}}{}%
	}%
}
\newcommand*{\da@xarrow}[7]{%
	\sbox0{$\ifx#7\scriptstyle\scriptscriptstyle\else\scriptstyle\fi#5#1#6\m@th$}%
	\sbox2{$\ifx#7\scriptstyle\scriptscriptstyle\else\scriptstyle\fi#5#2#6\m@th$}%
	\sbox4{$#7\dabar@\m@th$}%
	\dimen@=\wd0 %
	\ifdim\wd2 >\dimen@
	\dimen@=\wd2 %
	\fi
	\count@=2 %
	\def\da@bars{\dabar@\dabar@}%
	\@whiledim\count@\wd4<\dimen@\do{%
		\advance\count@\@ne
		\expandafter\def\expandafter\da@bars\expandafter{%
			\da@bars
			\dabar@ 
		}%
	}%
	\mathrel{#3}%
	\mathrel{%
		\mathop{\da@bars}\limits
		\ifx\\#1\\%
		\else
		_{\copy0}%
		\fi
		\ifx\\#2\\%
		\else
		^{\copy2}%
		\fi
	}%
	\mathrel{#4}%
}
\makeatother

\def\C{\mathbb{C}}
\def\Z{\mathcal{Z}}
\def\K{\mathbb{K}}

\def\L{\widetilde{L}}

\def\C{\mathcal{C}}
\def\M{\mathcal{M}^{\sigma}}

\newcounter{cnt1}
\newcounter{cnt2}
\newcounter{cnt3}

\newcommand{\blr}{\begin{list}{$($\roman{cnt1}$)$} {\usecounter{cnt1}
			\setlength{\topsep}{0pt} \setlength{\itemsep}{0pt}}}
\newcommand{\bla}{\begin{list}{$($\alph{cnt2}$)$} {\usecounter{cnt2}
			\setlength{\topsep}{0pt} \setlength{\itemsep}{0pt}}}
\newcommand{\bln}{\begin{list}{$($\arabic{cnt3}$)$} {\usecounter{cnt3}
			\setlength{\topsep}{0pt} \setlength{\itemsep}{0pt}}}
\newcommand{\el}{\end{list}}

\makeatletter
\renewcommand{\fnum@algorithm}{\fname@algorithm}
\makeatother

\theoremstyle{definition}
\newtheorem{Thm}{Theorem}[section]
\newtheorem{Lem}[Thm]{Lemma}

\newtheorem{Def}[Thm]{Definition}
\newtheorem{Exm}[Thm]{Example}
\newtheorem{Rem}[Thm]{Remark}
\newtheorem{Cor}[Thm]{Corollary}

\title{}
\author{}
\date{}
\usepackage{amssymb}

\usepackage{hyperref}
\hypersetup{
	colorlinks,
	citecolor=red,
	filecolor=black,
	linkcolor=blue,
	urlcolor=black
}

\begin{document}
\title{Automorphisms of extensions of Lie-Yamaguti algebras and Inducibility problem}

\author{Saikat Goswami}
\address{TCG Centres for Research and Education in Science and Technology, Institute for Advancing Intelligence, Salt Lake, Kolkata-700091, West Bengal, India.}
\address{RKMVERI, Belur, Howrah-711202, West Bengal, India.}
\email{saikatgoswami.math@gmail.com}

\author{Satyendra Kumar Mishra}

\author{Goutam Mukherjee}\footnote{Corresponding author email [A3]: goutam.mukherjee@tcgcrest.org; gmukherjee.isi@gmail.com}
\address{TCG Centres for Research and Education in Science and Technology, Institute for Advancing Intelligence, Salt Lake, Kolkata-700091, West Bengal, India.}
\address{Academy of Scientific and Innovative Research (AcSIR), Ghaziabad- 201002, India.}
\email{satyamsr10@gmail.com, goutam.mukherjee@tcgcrest.org, gmukherjee.isi@gmail.com}


\begin{abstract}
Lie-Yamaguti algebras generalize both the notions of Lie algebras and Lie triple systems. In this paper, we consider the inducibility problem for automorphisms of extensions of Lie-Yamaguti algebras. More precisely, given an abelian extension 
\[\begin{tikzcd}
	0 & V & {\L} & L & 0
	\arrow[from=1-1, to=1-2]
	\arrow["i", from=1-2, to=1-3]
	\arrow["p", from=1-3, to=1-4]
	\arrow[from=1-4, to=1-5]
\end{tikzcd}\] 
of a Lie-Yamaguti algebra $L$, we are interested in finding the pairs $(\phi, \psi)\in \mathrm{Aut}(V)\times \mathrm{Aut}(L)$, which are inducible by an automorphism in $\mathrm{Aut}(\L)$. We connect the inducibility problem to the $(2,3)$-cohomology of Lie-Yamaguti algebra. In particular, we show that the obstruction for a pair of automorphisms in $\mathrm{Aut}(V)\times \mathrm{Aut}(L)$ to be inducible lies in the $(2,3)$-cohomology group $\mathrm{H}^{(2,3)}(L,V)$. We develop the Wells exact sequence for Lie-Yamaguti algebra extensions, which relates the space of derivations, automorphism groups, and $(2,3)$-cohomology groups of Lie-Yamaguti algebras. As an application,  we describe certain automorphism groups of semi-direct product Lie-Yamaguti algebras. In the sequel, we apply our results to discuss inducibility problem for nilpotent Lie-Yamaguti algebras of index $2$. We give examples of infinite families of such nilpotent Lie-Yamaguti algebras and characterize the inducible pairs of automorphisms for extensions arising from these examples. Finally, we write an algorithm to find out all the inducible pairs of automorphisms for extensions arising from nilpotent Lie-Yamaguti algebras of index $2$.   

\end{abstract}

\maketitle

\noindent

\thispagestyle{empty}

\vspace{-1.0cm}

\medskip

\tableofcontents

\vspace{-1.0cm}

\section{\large{Introduction}} 

\smallskip

The notion of the Lie triple system was formally introduced as an algebraic object by N. Jacobson \cite{Jacobson}  in connection with problems arising from quantum mechanics. The example of a Lie triple system that arose from quantum mechanics is known as the Meson field, as described by N. Jacobson. The notion of Lie-Yamaguti algebras is a generalization of Lie triple systems and Lie algebras, derived from Nomizu's work on the invariant affine connections on homogeneous spaces in the 1950s. 

\smallskip

Let $M$ be a smooth manifold with a linear connection $\nabla$. For a given fixed point $e\in M$, there is a local multiplication $\mu$ at $e$ compatible with $\nabla$. If $M$ is a reductive homogeneous space $A/K$ with the canonical connection, due to K. Nomizu \cite{nomizu-inv}, then this local multiplication $\mu$ satisfies certain special property (cf. \cite{kikk-geo}). In particular, if M is a Lie group $A$
itself, then the canonical connection is reduced to the connection of \cite{Cartan} and
the local multiplication $\mu$ coincides with the multiplication of $A$ in local. 

\smallskip

Motivated by this fact, M. Kikkawa investigated the problem of the existence of a global differentiable binary system on a reductive homogeneous space $A/K,$ which coincides locally with the above geodesic local multiplication $\mu$. M. Kikkawa \cite{kikk-geo} observed that the problem is related to the canonical connection and to the general Lie triple system defined on the tangent space $T_eM.$ K. Yamaguti \cite{yama-lts} introduced the notion of {\it general Lie triple system} as a vector space equipped with a pair consisting of a bilinear and a trilinear operation and satisfying some intriguing relations in order to characterize the torsion and curvature tensors of Nomizu's canonical connection \cite{nomizu-inv}. In \cite{kikk-geo}, Kikkawa renamed the notion of general Lie triple system as {\it Lie triple algebra}. Kinyon and Weinstein \cite{kinyon-weinstein} observed that Lie triple algebras, which they called {\it Lie-Yamaguti algebras} in their paper, can be constructed from Leibniz algebras (non-anti-symmetric analog of Lie algebras introduced by J. L. Loday \cite{Loday}).

\smallskip

Lie-Yamaguti algebras have been studied by several authors in \cite{kikk-kill,kikk-solv,sagle-sim-anti, sagle-tot-geodesic, sagle-anti, sagle-anti-homo, yama-cohomo, sagle-homo-red, sagle-homo-holo, sagle-simple, zhang}. A. Sagle constructed some remarkable examples of Lie-Yamaguti algebras arising from reductive homogeneous spaces in differential geometry \cite{sagle-sim-anti}. Malcev algebras \cite{Malcev1,Malcev2, Malcev3,Malcev4} are non-associative algebras, which play a dominant role in the theory of Mufang loops generalizing the role of Lie algebras in the theory of Lie groups. Every Malcev algebra $(L,\langle \cdot, \cdot \rangle)$ becomes a Lie-Yamaguti algebra with the binary operation $[\cdot,\cdot]$ and the ternary operation $\{\cdot,\cdot,\cdot\}$ defined by 
\begin{equation}\label{Malcev to LY}
[x,y]= \langle x,y \rangle \quad \text{and} \quad \{x,y,z\} = \langle x,\langle y,z\rangle\rangle - \langle y,\langle x,z \rangle\rangle + \langle\langle x,y\rangle, z \rangle, \quad x,y,z\in L.
\end{equation}
 The following diagram shows how Lie-Yamaguti algebras are related to different non-commutative algebras; we refer to \cite{yama-malcev} for more details.
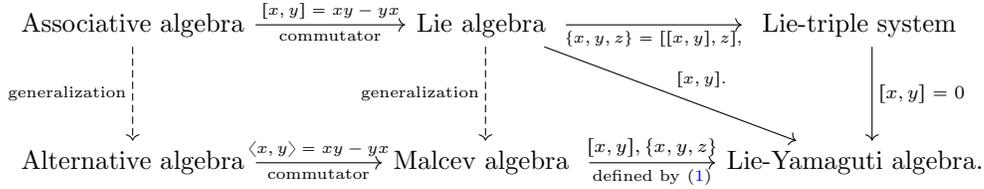
\[\begin{tikzcd}
	{\text{Associative algebra}} && {\text{Lie algebra }} && {\text{  Lie-triple system}} \\
	\\
	{\text{Alternative algebra}} && {\text{Malcev algebra }} && {\text{Lie-Yamaguti algebra.}}
	\arrow[from=1-1, to=1-3,"\text{\tiny $[x,y] = xy-yx$}","\text{\tiny commutator}"']
	\arrow[from=1-3, to=1-5,"\text{\tiny $\{x,y,z\}\\=[[x,y],z]$,}"']  
	\arrow[from=3-1, to=3-3, "\text{\tiny $\langle x,y\rangle = xy-yx$}","\text{\tiny commutator}"']
	\arrow[from=3-3, to=3-5,"\text{$[x,y],\{x,y,z\}$}","\text{\tiny defined by \eqref{Malcev to LY}}"']
	\arrow[dashed, from=1-3, to=3-3,swap, "\text{\tiny generalization}"]
	\arrow[dashed, from=1-1, to=3-1,swap, "\text{\tiny generalization}"]
	\arrow[from=1-5, to=3-5,shift left=1.5ex, "\text{$[x,y]=0$}"]
	\arrow[from=1-3, to=3-5, "\text{\tiny $[x,y].$}"]
\end{tikzcd}\]

\smallskip

Although some particular structure theoretic properties of Lie–Yamaguti algebras were studied in \cite{benito} and \cite{benito-irreducible}, the general structure theory of Lie-Yamaguti algebras is yet to be explored. Recently, in \cite{abdelwahab}, the authors have given a classification of nilpotent Lie-Yamaguti algebra up to dimension four. In this paper, we give examples of two infinite families of nilpotent Lie-Yamaguti algebras, which generalizes the nilpotent Lie-Yamaguti algebras $\mathcal{L}_{3,4}$ of \cite[Theorem 3.1]{abdelwahab}. It is well known that a Lie-Yamaguti algebra reduces to a Lie triple system (respectively, Lie algebra) when the binary (respectively, the ternary) operation is zero. In striking contrast, the infinite family $(\mathfrak{H}_n,[\cdot,\cdot],\{\cdot,\cdot,\cdot\})$ (refer to Example \ref{Counter_exm_1}) generalizing the $3$-dimensional Lie-Yamaguti algebra $\mathcal{L}_{3,4}$ of \cite{abdelwahab}, is simultaneously the $n$-dimensional Heisenberg Lie algebra and a Lie triple system, with both the operations being non-trivial. Therefore, we call it ``Heisenberg Lie-Yamaguti algebra''.

\medskip

K. Yamaguti introduced a cohomology theory for Lie-Yamaguti algebras in \cite{yama-cohomo} to study abelian extensions. This cohomology is used in \cite{zhang} to study infinitesimal deformations of Lie-Yamaguti algebras. In this paper, we interpret the inducibility of a pair of automorphisms in an abelian extension in terms of the cohomology theory of Lie-Yamaguti algebras. We show that the obstruction for a pair of automorphisms to be inducible lies in the $(2,3)$-cohomology of Lie-Yamaguti algebra. Later on, we develop the Wells map and the Wells exact sequence for Lie-Yamaguti algebra extensions. We refer to \cite{bar-singh,hazra-habib,jin,passi, robinson,wells} for the construction and applications of the Wells sequence for groups and Lie (super) algebras. 

\smallskip

The inducibility problem is described as follows: Let $0\rightarrow V\rightarrow \tilde{L}\rightarrow L\rightarrow 0$ be an abelian extension of Lie-Yamaguti algebras. If $\gamma\in \mathrm{Aut}(\tilde{L})$ is an automorphism such that $\gamma(v)\in V$, for all $v\in V$, then it induces a pair of automorphisms $(\gamma|_V, \bar{\gamma})\in \mathrm{Aut}(V)\times \mathrm{Aut}(L)$, see Section \ref{sec-4} for details. A pair $(\phi, \psi)\in \mathrm{Aut}(V)\times \mathrm{Aut}(L)$ is said to be inducible if there exists an automorphism $\gamma\in \mathrm{Aut}(\tilde{L})$ that induces the pair $(\phi,\psi)$. The main goal of this paper is to characterize these inducible pairs. 
We also discuss a necessary and sufficient condition of inducibility in terms of the Wells map. Finally, we discuss the inducibility problem for nilpotent Lie-Yamaguti algebras of index $2$ in section \ref{sec-6}. 
 
\smallskip
 
The organization of the paper is as follows: In Section \ref{sec-2}, we recall the definitions of Lie-Yamaguti algebras and their representations. We also recall the definition of the cohomology groups associated with Lie-Yamaguti algebras from \cite{yama-cohomo}. In Section \ref{sec-3}, we recall the notion of abelian extensions of Lie-Yamaguti algebras and their cohomological interpretation from \cite{yama-cohomo}. Section \ref{sec-4} and \ref{sec-5} deals with the inducibility problem for an abelian extension and describes the Wells exact sequence. As an application to the discussion in Section \ref{sec-5}, we describe certain automorphism groups of semi-direct product Lie-Yamaguti algebras. Finally, in Section \ref{sec-6}, we introduce two infinite families of nilpotent Lie-Yamaguti algebras (index $2$). We characterize inducible pairs of automorphisms for abelian extensions emerging from these Lie-Yamaguti algebras. We also give an algorithm to find all the inducible pairs of automorphisms for extensions arising from a (finite-dimensional) nilpotent  Lie-Yamaguti algebra of index $2$.

\bigskip

\section{\large Preliminaries} \label{sec-2}

\medskip

In this section, we recall the definition of Lie-Yamaguti algebras and their representations from \cite{yama-lts,yama-cohomo}. We also recall the cohomology theory of Lie-Yamaguti algebras from \cite{yama-cohomo}. Throughout this paper, we denote a field of characteristic zero by $\mathbb{K}$. We consider all vector spaces over the field $\K$ and linear maps to be $\mathbb{K}$-linear.

\subsection{Lie-Yamaguti algebras}
\begin{Def}\label{LYA}
A \textbf{Lie-Yamaguti algebra} is a triple $(L,[\cdot,\cdot],\{\cdot,\cdot,\cdot\})$, where $L$ is a vector space equipped with
\begin{itemize}
\item a $\mathbb{K}$-bilinear operation $[\cdot,\cdot]:L \times L \to L$ satisfying 
\begin{equation}
[a,b] = -[b,a] \tag{LY1}, \label{LY1}
\end{equation}
\item a $\mathbb{K}$-trilinear operation $\{\cdot,\cdot,\cdot\}:L \times L \times L \to L$ satisfying 
\begin{equation}
\{a,b,c\} = -\{b,a,c\}, \tag{LY2} \label{LY2}
\end{equation}
\end{itemize}
subject to the following conditions 
\begin{equation}
[[a,b],c]+\{a,b,c\}+[[b,c],a]+\{b,c,a\}+[[c,a],b]+\{c,a,b\}=0,
\tag{LY3} \label{LY3}
\end{equation}
\begin{equation}
\{[a,b],c,x\}+\{[b,c],a,x\}+\{[c,a],b,x\} = 0, \tag{LY4} \label{LY4}
\end{equation}
\begin{equation}
\{a,b,[x,y]\} = [\{a,b,x\},y]+[x,\{a,b,y\}], \tag{LY5} \label{LY5}
\end{equation}
\begin{equation}
\{a,b,\{x,y,z\}\} = \{\{a,b,x\},y,z\}+\{x,\{a,b,y\},z\}+\{x,y,\{a,b,z\}\}, \tag{LY6} \label{LY6} 
\end{equation}
\end{Def}
\noindent
for all $a,b,c,x,y,z \in L$. A Lie-Yamaguti algebra is said to be \textbf{abelian} if both the brackets $[\cdot,\cdot]$ and $\{\cdot,\cdot,\cdot\}$ are trivial. We denote a Lie-Yamaguti algebra $(L, [\cdot,\cdot], \{\cdot,\cdot,\cdot\})$ simply by $L$.  

\medskip

\begin{Def}
Let $L$ and $L'$ be two Lie-Yamaguti algebras. A linear map  $\varphi:L \to L'$ is a \textbf{morphism of Lie-Yamaguti algebras} if 
\begin{align*}
\varphi([a,b]) = [\varphi(a),\varphi(b)] \quad\mbox{and }\quad \varphi(\{a,b,c\}) = \{\varphi(a),\varphi(b),\varphi(c)\}, \quad \forall~ a,b,c \in L.
\end{align*}
It is said to be an isomorphism if $\phi$ is invertible. 
\end{Def}

We denote by $\mathrm{Aut}(L)$ the set of all automorphisms (self-isomorphisms) of a Lie-Yamaguti algebra $L$. Then $\mathrm{Aut} (L)$ has an obvious group structure called the automorphism group. We list some natural examples of Lie-Yamaguti algebras. 

\begin{Exm} 
	Let $(L,[\cdot,\cdot])$ be a Lie algebra. Define a binary and ternary bracket on $L$ by
	\[
	[a,b]:=[a,b], \quad \{a,b,c\} := [[a,b],c],\quad \forall a,b,c\in L. 
	\]
	Then $(L,[\cdot,\cdot],\{\cdot,\cdot,\cdot\})$ is a Lie-Yamaguti algebra.
\end{Exm}
\begin{Exm}
	Let $(L,[ \cdot,\cdot ])$ be a reductive Lie algebra with decomposition $L=G \oplus H$, such that $[G,H] \subseteq H$ and $[G,G] \subseteq G$. Let $\pi_{H}:L \to H$ and $\pi_{G}:L \to G$ be the projection maps onto $H$ and $G$ respectively. Then $H$ is a Lie-Yamaguti algebra with respect to the brackets
	\[
		[a,b] := \pi_{H}[a,b], 
		\quad \{a,b,c\} := [\pi_{G}[ a,b],c],\quad\forall ~a,b,c\in H.
	\]
\end{Exm}

\begin{Exm}
Let $(L,\cdot)$ be a Leibniz algebra. Then $L$ is a Lie-Yamaguti algebra with respect to the following binary and ternary bracket
\[[a,b]  =  a\cdot b - b \cdot a,  \quad
\{a,b,c\}  =  -(a \cdot b) \cdot c, \quad\forall ~ a,b,c\in L.\]
\end{Exm}

\begin{Exm}
Let $M$ be a closed manifold (a compact smooth manifold without boundary) with an affine connection. Let $\mathfrak{X}(M)$ denote the set of all vector fields on $M$. Define a binary and a ternary bracket on $\mathfrak{X}(M)$ as follows
\[
[x,y] = - T(x,y), \quad \{x,y,z\} = -R(x,y)z,
\]
where $T$ is the torsion and $R$ is the  curvature tensor. Then $(\mathfrak{X}(M),[\cdot,\cdot],\{\cdot,\cdot,\cdot\})$ is a Lie-Yamaguti algebra. 
\end{Exm}

\subsection{Representations of Lie-Yamaguti algebras}
\begin{Def}
Let $L$ be a Lie-Yamauti algebra and $V$ be a vector space. A \textbf{representation} of $L$ on $V$ consists of a linear map $\rho:L \to \operatorname{End}(V)$ and two bilinear maps $D,\theta:L \times L \to \operatorname{End}(V)$ such that 
\begin{align}
&D(a,b) + \theta(a,b) - \theta(b,a) = [\rho(a),\rho(b)] - \rho([a,b]), \label{R1} \tag{R1} \\
&\theta(a,[b,c]) - \rho(b) \theta(a,c) + \rho(c) \theta(a,b) = 0, \label{R2} \tag{R2} \\
&\theta([a,b],c) - \theta(a,c) \rho(b) + \theta(b,c) \rho(a) = 0,
\label{R3} \tag{R3} \\
&\theta(c,d) \theta(a,b) - \theta(b,d) \theta(a,c) - \theta(a,\{b,c,d\}) + D(b,c) \theta(a,d) = 0, 
\label{R4} \tag{R4} \\
&[D(a,b),\rho(c)] = \rho(\{a,b,c\}),
\label{R5} \tag{R5} \\
&[D(a,b),\theta(c,d)] = \theta(\{a,b,c\},d) + \theta(c,\{a,b,d\}), \label{R6} \tag{R6}
\end{align}
for all $a,b,c,d \in L$.
\end{Def}
\noindent
 Using relations (\ref{R1}), (\ref{R2}), (\ref{R3}), and (\ref{R5}) stated above, we get an additional identity 
\begin{equation} \tag{R7}
D([a,b],c) + D([b,c],a) + D([c,a],b) = 0.
\end{equation}

\smallskip 

\noindent
Taking $V=L$, there is a natural representation of $L$ on itself, called the \textbf{adjoint representation} of $L$, where the representation maps $\rho, D, \theta$ are given by 
\begin{equation*}
\rho(a)(b) := [a,b], \quad D(a,b)(c) := \{a,b,c\}, \quad\mbox{and }\quad \theta(a,b)(c) :=  \{c,a,b\}.
\end{equation*}

\noindent
One can characterize representations of Lie-Yamaguti algebras using semi-direct product Lie-Yamaguti algebras (see \cite{LieYReppairs}).

\begin{Lem}
Let $L$ be a Lie-Yamaguti algebra and $V$ be a vector space. Let $\rho:L \to \mathrm{End}(V)$ and $D,\theta:L \times L \to \mathrm{End}(V)$ be linear and bilinear maps respectively. Then $(\rho,\theta,D;V)$ is a representation of $L$ if and only if there is a Lie-Yamaguti algebra structure on the vector space $L \oplus V$, with the binary and ternary brackets defined by 
\begin{eqnarray*}
[a+u,b+v]_{\ltimes} & := & [a,b] + \rho(a)(v) - \rho(b)(u), \\
\{a+u,b+v,c+w\}_{\ltimes} & := & \{a,b,c\} + D(a,b)(w) + \theta(b,c)(u) - \theta(a,c)(v). 
\end{eqnarray*}
This Lie-Yamaguti algebra is called the \textbf{semi-direct product Lie-Yamaguti algebra}, denoted by $L \ltimes V$.
\end{Lem}
 
\begin{Def} 
Let $L$ be a Lie Yamaguti algebra and $(\rho,D,\theta; V)$ and $(\rho',D',\theta'; V')$ be two representations of $L$. A linear map $\varphi:V \to V'$ is said to be a \textbf{morphism of representations} of $L$ if for any $a,b \in L$ and $v \in V$,
\[\varphi (\rho(a) (v)) = \rho'(a)(\varphi(v)), \quad \varphi D(a,b) (v) = D'(a,b)(\varphi(v)), \quad \mbox{and }\quad \varphi \theta(a,b)(v) = \theta'(a,b)(\varphi(v)).\]
The morphism $\varphi$ is said to be an isomorphism of representations if $\varphi$ is invertible. 
\end{Def}

\medskip

\subsection{A Cochain complex of Lie-Yamaguti algebras}

We recall the cochain complex associated with Lie-Yamaguti algebra from \cite{yama-cohomo}. 
Let $(\rho,D,\theta;~V)$ be a representation of a Lie-Yamaguti algebra $L$. Then, the cochain complex is described as follows: 

\medskip

\begin{itemize}
\item Set $C^1(L,V) = \operatorname{Hom}~ (L,V)$ and  $C^0(L,V)$ be the subspace spanned by the diagonal elements $(f,f) \in C^1(L,V) \times C^1(L,V)$. 
\smallskip
\item For $n \ge 2$, denote by $C^n(L,V)$, the space of all $n$-linear maps $f \in \operatorname{Hom}~(L^{\otimes n},V)$ satisfying
\[
	f(x_1,\cdots,x_{2i-1},x_{2i},\cdots,x_{n}) = 0, \quad\mbox{if }  x_{2i-1} = x_{2i}, \quad \forall~i=1,2,\ldots, \lfloor n/2 \rfloor.
\]
Then, for $p \ge 1$, set
\[
	C^{(2p,2p+1)}(L,V) := C^{2p}(L,V) \times C^{(2p+1)}(L,V).
\]
\end{itemize}

Any element $(f,g) \in C^{(2p,2p+1)}(L,V)$ is called a $(2p,2p+1)$-cochain. Now, we define the coboundary in the following cochain complex of $L$ with coefficients in $V$

\[\begin{tikzcd}
	{C^0(L,V)} & {C^{(2,3)}(L,V)} & {C^{(4,5)}(L,V)} & {C^{(6,7)}(L,V)} & \cdots \\
	& {C^{(3,4)}(L,V)}
	\arrow["\delta", from=1-2, to=1-3]
	\arrow["\delta", from=1-3, to=1-4]
	\arrow["\delta", from=1-4, to=1-5]
	\arrow["{\delta^*}", from=1-2, to=2-2]
	\arrow["{\delta}", from=1-1, to=1-2]
\end{tikzcd}\]

\smallskip
\noindent
For $p \ge 1$, the coboundary map $\delta=(\delta_I,\delta_{II}):C^{(2p,2p+1)}(L,V) \to C^{(2p+2,2p+3)}(L,V)$ is defined by 
\medskip
\begin{align*}
	\delta&_If(x_1,x_2,\ldots,x_{2p+2}) \\
	&=
	(-1)^p 
	\Big[\rho(x_{2p+1})g(x_1,\ldots,x_{2p},x_{2p+2}) - \rho(x_{2p+2})g(x_1,\ldots,x_{2p+1}) 
	- g(x_1,\ldots,x_{2p},[x_{2p+1},x_{2p+2}])\Big] \\
	& \quad + \sum_{k=1}^p (-1)^{k+1} D(x_{2k-1},x_{2k}) f(x_1,\ldots,\hat{x}_{2k-1}, \hat{x}_{2k},\ldots,x_{2p+2}) \\
	&\quad + \sum_{k=1}^p \sum_{j=2k+1}^{2p+2} (-1)^k f(x_1,\ldots,\hat{x}_{2k-1}, \hat{x}_{2k},\ldots,\{x_{2k-1},x_{2k},x_j\}, \ldots, x_{2p+2}) \quad \mbox{and}\\
	\\
	\delta&_{II}g(x_1,x_2,\ldots,x_{2p+3}) \\
	&=
	(-1)^p
	\Big[ \theta(x_{2p+2},x_{2p+3})g(x_1,\ldots,x_{2p+1}) - \theta(x_{2p+1},x_{2p+3})g(x_1,\ldots,x_{2p},x_{2p+2})\Big]	\\
	& \quad
	+ \sum_{k=1}^{p+1} (-1)^{k+1} D(x_{2k-1},x_{2k})g(x_1,\ldots,\hat{x}_{2k-1},\hat{x}_{2k},\ldots,x_{2p+3}) \\
	& \quad
	+ \sum_{k=1}^{p+1} \sum_{j=2k+1}^{2p+3} (-1)^k g(x_1,\ldots,\hat{x}_{2k-1},\hat{x}_{2k},\ldots, \{x_{2k-1},x_{2k},x_j\},\ldots,x_{2p+3}).
\end{align*}

\medskip
\noindent
In particular, for $p=1$, the coboundary map $\delta:C^{(2,3)}(L,V) \to C^{(4,5)}(L,V)$ is given by 
\begin{multline}
	\delta_If(a,b,c,d) \label{coboundary_map 1}
	=
	- \rho(c)g(a,b,d) + \rho(d)g(a,b,c) 
	+ g(a,b,[c,d]) \\
	+ D(a,b) f(c,d)
	- f(\{a,b,c\},d) - f(c,\{a,b,d\}), \quad
\end{multline}
\begin{multline}
	\delta_{II}g(a,b,c,d,e) \label{coboundary_map 2}
	=
	- \theta(d,e)g(a,b,c) + \theta(c,e)g(a,b,d) 
	+ D(a,b)g(c,d,e) - D(c,d)g(a,b,e)\\
	\qquad \qquad \qquad \qquad \qquad 
	- g(c,\{a,b,d\},e) - g(c,d,\{a,b,e\}) 
	+ g(a,b,\{c,d,e\}) 
	- g(\{a,b,c\},d,e). 
\end{multline}

\medskip
\noindent
The map $\delta = (\delta_I,\delta_{II}):C^0(L,V) \to C^{(2,3)}(L,V)$ is defined by
\begin{equation} \label{coboundary_map 3}
	\delta_I f(a,b) = \rho(a)f(b) - \rho(b)f(a) - f([a,b]),
\end{equation}
\begin{equation} \label{coboundary_map 4}
	\delta_{II}f(a,b,c) = \theta(b,c)f(a) - \theta(a,c)f(b) + D(a,b)f(c) - f(\{a,b,c\}).
\end{equation}

\medskip
\noindent
Finally, let us define the map $\delta^{\star} = (\delta^{\star}_I,\delta^{\star}_{II}):C^{(2,3)}(L,V)\to C^{(3,4)}(L,V)$ by
\begin{equation}  \label{coboundary_map 5}
	\begin{split}
		\delta^{\star}_If(a,b,c) 
		&= -\rho(a)f(b,c) - \rho(b)f(c,a) - \rho(c)f(a,b) + f([a,b],c) + f([b,c],a) \\
		& \quad \quad
		+ f([c,a],b) + g(a,b,c) + g(b,c,a) + g(c,a,b),
	\end{split}
\end{equation} 
\begin{equation} \label{coboundary_map 6}
	\begin{split}
		\delta^{\star}_{II}g(a,b,c,d) 
		&= \theta(a,d)f(b,c) + \theta(b,d)f(c,a) + \theta(c,d)f(a,b) + g([a,b],c,d) \\ 
		&\quad \quad
		+ g([b,c],a,d) + g([c,a],b,d). 
	\end{split}
\end{equation}

\bigskip
\noindent
\textbf{Cohomology groups of Lie-Yamaguti algebras:}
Let us denote the set of all $(2p,2p+1)$-cocycles

\smallskip 

\noindent by $Z^{(2p,2p+1)}(L,V)$ and the set of all $(2p,2p+1)$-coboundaries by $B^{(2p,2p+1)}(L,V)$. For $p = 1$,
	\[Z^{(2,3)}(L,V) = \{(f,g) \in C^{(2,3)}(L,V): \delta_If = 0 = \delta_I^\star f ~ \text{and}~ \delta_{II} g = 0 = \delta_{II}^{\star} g\}, \] 
	\[B^{(2,3)}(L,V) = \{f \in C^1(L,V) : \delta_I f = 0 = \delta_{II} f\}.\]
	For $p \ge 2$,
	\[Z^{(2p,2p+1)}(L,V) = \{(f,g) \in C^{(2p,2p+1)}(L,V) : \delta_If=0=\delta_{II}g\},\] 
	\[B^{(2p,2p+1)}(L,V) = \{\delta(f,g) : (f,g) \in C^{(2p-2,2p-1)}(L,V)\}.\] 

\medskip

\noindent
The cohomology groups of $L$ with coefficients in $V$ is defined as follows:
	\[
	H^1(L,V) := \{f \in C^1(L,V) : \delta_If = 0 = \delta_{II} f\}.
	\]
The group $H^1(L,V)$ is called the first cohomology group. Let $\mathrm{Der}(L,V)$ be the space of all derivations from $L$ to $V$. Then, $H^1(L,V) = \mathrm{Der}(L,V)$. For $p \ge 1$, denote 
	\[
	H^{(2p,2p+1)}(L,V):= \frac{Z^{(2p,2p+1)}(L,V)}{B^{(2p,2p+1)}(L,V)}.
	\]
The group $\mathrm{H}^{(2p,2p+1)}(L,V)$ is called the $(2p,2p+1)$ cohomology group of $L$ with coefficients in $V$. 	
	
\bigskip

\section{\large Abelian extensions of Lie-Yamaguti algebras } \label{sec-3}

\smallskip

In this section, we recall a cohomological interpretation of abelian extensions of Lie-Yamaguti algebras from \cite{yama-cohomo}. Let $L$ be a Lie-Yamaguti algebra.

\begin{Def}
A subspace $W$ of $L$ is an \textbf{ideal} of $L$ if $$[W,L] \subseteq W,\quad \{W,L,L\} \subseteq W, \quad\text{and}\quad \{L,L,W\} \subseteq W.$$ 
\end{Def}

\begin{Def}
An ideal $W$ of $L$ is said to be an \textbf{abelian ideal} in $L$ if $$ [W,W] = (0),\quad \mbox{and }\quad\{W,W,L\} = \{W,L,W\} = \{L,W,W\} = (0).$$ 
\end{Def}

\begin{Def}
The \textbf{center} $\Z(L)$ of $L$ is defined as 
\[\Z(L):= \{a \in L : [a,x]=0,~ \{a,x,y\}=0=\{x,y,a\}~\text{for all}~x,y \in L\}.\]
\end{Def}

\begin{Def} Let $L$ and $V$ be Lie Yamaguti algebras. An \textbf{extension} of $L$ by $V$ is a Lie-Yamaguti algebra $\L$, which fits into a short exact sequence in the category of Lie-Yamaguti algebras over $\mathbb{K}$ 
\[\begin{tikzcd}
0 & V & {\L} & L & 0
\arrow[from=1-1, to=1-2]
\arrow["i", from=1-2, to=1-3]
\arrow["p", from=1-3, to=1-4]
\arrow[from=1-4, to=1-5].
\end{tikzcd}\] 
\end{Def}
\noindent
We denote an extension as above simply by ${\L}$. An extension is called an \textbf{abelian extension} if $V$ is an abelian ideal in $\L$.

\smallskip

Throughout the article, by a \textbf{section} of $p:\L \to L$ we mean a linear map $s: L \to \L$  such that $p \circ s = id_L$. Since any short exact sequence of vector spaces is split, there exists a section $s:L \to \L$ of $p:\L \to L$. Hence, as a vector space $\L = L \oplus V$. 

\begin{Def}
Let ${\L}$ and ${\widehat{L}}$ be two abelian extensions of $L$ by $V$. They are said to be \textbf{equivalent} if there is a morphism $\varphi:\L \to \widehat{L}$ of Lie-Yamaguti algebra making the following diagram commutative 
\begin{align*}
\xymatrix{
0 \ar[r] & V  \ar@{=}[d] \ar[r]^{i} & \L \ar[d]^\varphi \ar[r]^{p} & L  \ar@{=}[d] \ar[r] & 0 
\\
0 \ar[r] & V \ar[r]_{i'} & \widehat{L} \ar[r]_{p'} & L \ar[r] & 0.}
\end{align*}	 
\end{Def}
Let $ 0 \to V \xrightarrow[]{i} \L \xrightarrow[]{p} L \to 0$ be an abelian extension of $L$ by $V$ and $s$ be a section of $p$. Then, there is a representation of $L$ on $V$ with the representation maps $\rho, D,$ and $\theta$ defined as follows: for any $a,b \in L$ and $v \in V$, 
\begin{align}
	\rho(a)(v) &:= [s(a),v]_{\L}, \label{rep1}\\ 
	D(a,b)(v) &:= \{s(a), s(b),v\}_{\L}, \label{rep2}\\
	\theta(a,b)(v) &:= \{v, s(a), s(b)\}_{\L}. \label{rep3}
\end{align}
The above representation of $L$ on $V$ is called the \textbf{induced representation of $L$ on $V$}. A direct computation shows that the representation maps does not depend on the choice of the section. 

\bigskip
We next recall that there is a $(2,3)$-cocycle associated to the abelian extension ${\L}$ given by $(\alpha,\beta)$ consisting of a bilinear map $\alpha: L \times L \to V$ and a trilinear map $\beta: L \times L \times L \to V$ given by
\begin{equation}\label{ind_cocycle-I}
\alpha(a,b) := [s(a),s(b)]_{\L} - s([a,b]),
\end{equation}
\begin{equation}\label{ind_cocycle-II}
\beta(a,b,c) := \{s(a),s(b),s(c)\}_{\L} - s(\{a,b,c\}),\quad \forall a,b,c \in L.
\end{equation}
It is easy to verify that $(\alpha,\beta) \in C^{(2,3)}(L,V)$. In fact, the following lemma below shows that it is a $(2,3)$-cocycle, which is called the \textbf{induced cocycle} obtained from the extension ${\L}$.

\begin{Lem} \label{Lem_induced-cocycle}
	The above defined $(\alpha,\beta) \in C^{(2,3)}(L,V)$ is a $(2,3)$-cocycle. 
\end{Lem}

\begin{proof}
To show that $(\alpha,\beta)$ is a $(2,3)$-cocycle, we need to show that 
\[\delta(\alpha,\beta)= 0\quad\text{and}\quad \delta^*(\alpha,\beta) = 0,~i.e.,\]  	
\[\delta_I\alpha = 0,\quad \delta_{II}\beta = 0 \quad\text{and}\quad \delta_I^*\alpha=0,\quad \delta_{II}^*\beta=0.\]
Recall the representation maps $\rho, D,$ and $\theta$ of the induced representation are
\[\rho(a)(v) = [s(a),v]_{\L},\quad D(a,b)(v) = \{s(a), s(b),v\}_{\L}, \quad \mbox{and}\]
\[\theta(a,b)(v) = \{v, s(a), s(b)\}_{\L},\quad \forall~a,b\in L, v\in V.\]
By the definitions of $\delta$ and $\delta^*$ we get 
\begin{align*}
\delta_I\alpha(a_1,&a_2,a_3,a_4) \\
&=
		- \rho(a_3)\beta(a_1,a_2,a_4) 
		+ \rho(a_4)\beta(a_1,a_2,a_3) 
		+ \beta(a_1,a_2,[a_3,a_4]) \\
		&\quad \quad
		+ D(a_1,a_2) \alpha(a_3,a_4)
		- \alpha(\{a_1,a_2,a_3\},a_4) - \alpha(a_3,\{a_1,a_2,a_4\}) \\
		&= 
		0 \quad \quad\quad \quad\quad \quad\text{(expanding and using (LY5) we get zero)}.
		\\ ~~ \\
		\delta_{II}\beta(a_1,&a_2,a_3,a_4,a_5) \\
		&=
		- \theta(a_4,a_5)\beta(a_1,a_2,a_3) 
		+ \theta(a_3,a_5)\beta(a_1,a_2,a_4)
		+ D(a_1,a_2)\beta(a_3,a_4,a_5) \\
		&\quad \quad
		- D(a_3,a_4)\beta(a_1,a_2,a_5) 
		- \beta(\{a_1,a_2,a_3\},a_4,a_5) 
		- \beta(a_3,\{a_1,a_2,a_4\},a_5) \\
		&\quad \quad
		- \beta(a_3,a_4,\{a_1,a_2,a_5\})  
		+ \beta(a_1,a_2,\{a_3,a_4,a_5\}) \\
		&=
		0 \quad \quad\quad \quad\quad \quad\text{(expanding and using (LY6) we get zero)}.
		\end{align*}
				\begin{align*}
			\delta_I^*\alpha(a_1,&a_2,a_3) \\
		&= - \sum_{\circlearrowleft(a_1,a_2,a_3)}\rho(a_1)\alpha(a_2,a_3) 
		~+ \sum_{\circlearrowleft(a_1,a_2,a_3)}\alpha([a_1,a_2],a_3)
		~+ \sum_{\circlearrowleft(a_1,a_2,a_3)} \beta(a_1,a_2,a_3) \\
		&=
		0 \quad \quad\quad \quad\quad \quad\text{(expanding and using (LY3) we get zero)}.
		\\ ~~ \\
		\delta_{II}^* \beta(a_1,&a_2,a_3,a_4) \\
		&= 
		\theta(a_1,a_4)\alpha(a_2,a_3) + \theta(a_2,a_4)\alpha(a_3,a_4) + \theta(a_3,a_4)\alpha(a_1,a_2) \\
		&\quad \quad
		+ \beta([a_1,a_2],a_3,a_4) + \beta([a_2,a_3],a_1,a_4) + \beta([a_3,a_1],a_2,a_4) \\	
		&= 0 \quad \quad\quad \quad\quad \quad\text{(expanding and using (LY4) we get zero)}.
	\end{align*}
	Thus, $(\alpha,\beta)$ is a $(2,3)$-cocycle.  
\end{proof}

The following lemma shows that the cohomology class of the $(2,3)$-cocycle $(\alpha,\beta)$ induced by an abelian extension ${\L}$ does not depend on the choice of the section.

\begin{Lem} \label{Lem_ind of sec}
	Let $s$ and $t$ be two sections of $p$ and the corresponding induced cocycles are given by  
	\[
	\alpha_s(a,b) = [s(a),s(b)]_{\L} - s([a,b]),\quad
	\beta_s(a,b,c) = \{s(a),s(b),s(c)\}_{\L} - s(\{a,b,c\}),
	\]
	and
	\[
	\alpha_t(a,b) = [t(a),t(b)]_{\L} - t([a,b]),\quad
	\beta_t(a,b,c) = \{t(a),t(b),t(c)\}_{\L} - t(\{a,b,c\}).
	\]
	Then, the coycles $(\alpha_s,\beta_s)$ and $(\alpha_t,\beta_t)$ are cohomologous, i.e., $(\alpha_s,\beta_s) - (\alpha_t,\beta_t) \in B^{(2,3)}(L,V)$.
\end{Lem}
\begin{proof}
	
	Since $ps = pt = 1$, we have $s(a) - t(a) \in V$ for all $a \in L$. Let us define a map $\lambda:L \to V$ by 
	\[
		\lambda(a) := s(a) - t(a).
	\]
	and do the following computation 
	\begin{align*}
		\alpha_s(a,b) - \alpha_t(a,b) 
		&= [s(a),s(b)]_{\L} - s[a,b] - [t(a),t(b)]_{\L} + t[a,b] \\
		&= [\lambda(a) + t(a),\lambda(b) + t(b)]_{\L} - [t(a),t(b)]_{\L} - \lambda[a,b] \\
		&= [t(a),\lambda(b)]_{\L} - [t(b),\lambda(a)]_{\L} - \lambda[a,b] \\
		&= \rho(a)(\lambda(b)) - \rho(b) (\lambda(a)) - \lambda[a,b] \\
		&= \delta_I(\lambda)(a,b).
	\end{align*}
	A similar computation shows that $\beta_s(a,b,c) - \beta_t(a,b,c) = \delta_{II}(\lambda) (a,b,c)$. Thus, it follows that
	\[
		 \delta (\lambda) = (\alpha_s,\beta_s) - (\alpha_t,\beta_t) \in B^{(2,3)}(L,V).  
	\]
	Hence completing the proof.
\end{proof}

\bigskip

\section{\large{Inducibility of a pair of automorphisms}} \label{sec-4}

\medskip

This section considers the inducibility problem for an abelian extension of Lie-Yamaguti algebras. The two main results Lemma \ref{Lem_compatibility condition} and Theorem \ref{Thm_cmptbl_indcbl} of this section provide a necessary and sufficient condition for a pair of Lie-Yamaguti algebra automorphisms to be inducible in an abelian extension.

\medskip 
 
Let $\L$ be an abelian extension of $L$ by $V$. Restricting any automorphism $\gamma \in \mathrm{Aut}(\L)$ to $V$ gives an automorphism of $V$ if and only if $\gamma(V) \subseteq V$. We define $$\mathrm{Aut}_V(\L) := \{ \gamma \in \mathrm{Aut}(\L): \gamma(V) = V\}.$$ 
Let $s$ be a section of $p$, then any $\gamma \in \mathrm{Aut}_V(\L)$ induces a pair $(\gamma|_V,\widebar{\gamma}) \in \mathrm{Aut}(V) \times \mathrm{Aut}(L)$, where 
\[\widebar{\gamma}(a) := p\gamma s(a), \quad \forall~ a \in L.\]   
Thus, we have the following diagram,
\[\begin{tikzcd}
	0 & V & {\tilde{L}} & L & 0~ \\
	0 & V & {\tilde{L}} & L & 0.
	\arrow[from=1-1, to=1-2]
	\arrow["i", from=1-2, to=1-3]
	\arrow["p", from=1-3, to=1-4]
	\arrow[from=1-4, to=1-5]
	\arrow[from=2-1, to=2-2]
	\arrow["i"', from=2-2, to=2-3]
	\arrow["p"', from=2-3, to=2-4]
	\arrow[from=2-4, to=2-5]
	\arrow["{\gamma|_V}", from=1-2, to=2-2]
	\arrow["\gamma", from=1-3, to=2-3]
	\arrow["{\bar{\gamma}}", from=1-4, to=2-4]
\end{tikzcd}\]

 Even though $s$ is not a Lie-Yamaguti algebra morphism, it can be shown that  $\bar{\gamma}$ is a Lie-Yamaguti algebra automorphism of $L$. As a result, the above construction gives rise to a group homomorphism
\[\tau: \mathrm{Aut}_V(\L) \to \mathrm{Aut}(V) \times \mathrm{Aut}(L), \qquad \tau(\gamma) := (\gamma|_V,\bar{\gamma}).\] 
The definition of $\widebar{\gamma}$ may depend on the choice of $s$. However, the following lemma ensures that the construction of $\widebar{\gamma}$ does not depend on the choice of $s$.

\begin{Lem}
The definition of $\bar{\gamma}$ does not depend on the choice of a section. 
\end{Lem}

\begin{proof}
Let $s$ and $t$ be two sections of $p$. Then, we have the following expression,
\[ps(a) - pt(a) = 0 \quad \implies \quad s(a) - t(a) \in Ker(p) = V, \quad \forall~ a\in L,~v\in V.\]
Since $\gamma \in \mathrm{Aut}_V(\L)$, we have $\gamma(s(a) - t(a)) \in V$. Thus, $p \gamma s (a) - p \gamma t (a) = 0$, which implies that $p \gamma s(a) = p \gamma t(a)$. 
\end{proof} 

\begin{Def}
A pair $(\phi,\psi) \in \mathrm{Aut}(V) \times \mathrm{Aut}(L)$ is said to be an \textbf{inducible pair} if $(\phi,\psi) \in \tau\big(\mathrm{Aut}_V(\L)\big)$  
\end{Def}

\noindent
It is natural to ask the following question: 

\smallskip 
\noindent
\textbf{When is a pair of Lie-Yamaguti algebra automorphisms $(\phi, \psi) \in \mathrm{Aut}(V) \times \mathrm{Aut}(L)$ inducible?}

\medskip 

As a first step towards answering this question, we find a necessary condition for a pair to be inducible. Let $\L$ be an abelian extension of $L$ by $V$,
\[\begin{tikzcd}
0 & V & {\L} & L & 0
\arrow[from=1-1, to=1-2]
\arrow["i", from=1-2, to=1-3]
\arrow["p", from=1-3, to=1-4]
\arrow[from=1-4, to=1-5].
\end{tikzcd}\]
From Section \ref{sec-3}, we have the induced representation $(\rho,D,\theta ; V)$ of $L$ on $V$, where the representation maps are
\[\rho(a)(v) = [s(a),v]_{\L}, \quad D(a,b)(v) = \{s(a), s(b),v\}_{\L},\quad\text{and} \quad \theta(a,b)(v) = \{v, s(a), s(b)\}_{\L}.\]
Now, for any given pair $(\phi,\psi) \in \mathrm{Aut}(V) \times \mathrm{Aut}(L)$, we introduce a new representation $(\rho',D',\theta' ; V)$ of $L$ on $V$ as follows
\[\rho'(a)(v) := [s\psi(a), v]_{\L}, \quad D'(a,b)(v) := \{s\psi(a), s\psi(b), v\}_{\L},~\text{and} \quad
\theta'(a,b)(v) := \{v, s\psi(a), s\psi(b)\}_{\L},\]
for all $a,b\in L,~v\in V$. We refer to this representation as the \textbf{twisted representation} of $L$ on $V$ by $\psi$. 

\smallskip

\begin{Lem} \label{Lem_compatibility condition}
If a pair $(\phi,\psi) \in \mathrm{Aut}(V) \times \mathrm{Aut}(L)$ is inducible, then $\phi: (\rho,D,\theta; V) \to (\rho',D',\theta';V)$ is a morphism of representations of $L$, i.e., 
\[\phi \rho(a) (v) = \rho'(a) \phi (v), \quad \phi D(a,b) (v) = D'(a,b) \phi (v), \quad \phi \theta(a,b) (v) = \theta'(a,b) \phi (v),\]
for all $a,b \in L$ and $v \in V$.
\end{Lem} 
\begin{proof}
We are given that $(\phi,\psi)$ is an inducible pair, i.e., there exists $\gamma \in \mathrm{Aut}_V(\L)$ such that $\tau(\gamma) = (\gamma|_V,\bar{\gamma}) = (\phi, \psi)$. For any $a,b \in L$ and $v \in V$, we have
\[\phi \rho(a)(v) = \phi[s(a),v]_{\L} = [\gamma s (a), \phi(v)]_{\L}, \quad \rho'(a) \phi (v) = [s\psi(a),\phi(v)]_{\L} = [sp\gamma s(a),\phi(v)]_{\L}.\]
Note that $p\left(sp\gamma s(a) - \gamma s(a)\right) = 0$ implies that $sp\gamma s(a) - \gamma s(a) \in V$. Since $[V,V]_{\L} = (0)$ we have 
\[\rho'(a) \phi (v) - \phi \rho(a)(v) = [sp\gamma s(a) - \gamma s(a) , \phi(v)]_{\L} = 0,\] 
\noindent
which gives us $\phi \rho(a)(v) = \rho'(a) \phi (v)$. By a similar computation, we obtain the other two identities
\[\phi D(a,b) (v) = D'(a,b) \phi(v), \quad \mbox{and}\quad\phi \theta(a,b) (v) = \theta'(a,b) \phi(v)\]
Hence, the map $\phi: V \to V$ is a morphism of representations.
\end{proof}

\noindent
The above lemma gives us a necessary condition for inducibility. Note that one can also express the necessary condition of inducibility obtained above in terms of the following commutative diagram 
\[\begin{tikzcd}
L & L & {L \times L} & {L \times L} & {L \times L} & {L \times L} \\
{End(V)} & {End(V)} & {End(V)} & {End(V)} & {End(V)} & {End(V)}
\arrow["\xi"', from=2-1, to=2-2]
\arrow["\rho"', from=1-1, to=2-1]
\arrow["\rho", from=1-2, to=2-2]
\arrow["{(\psi,\psi)}", from=1-3, to=1-4]
\arrow["\xi"', from=2-3, to=2-4]
\arrow["D"', from=1-3, to=2-3]
\arrow["D", from=1-4, to=2-4]
\arrow["{(\psi,\psi)}", from=1-5, to=1-6]
\arrow["\xi"', from=2-5, to=2-6]
\arrow["\theta"', from=1-5, to=2-5]
\arrow["\theta", from=1-6, to=2-6]
\arrow["\psi", from=1-1, to=1-2],
\end{tikzcd}\]
where $\xi$ is defined by $\xi(f) := \phi \circ f \circ \phi^{-1}$. 
With this necessary condition, we define the notion of ``compatible pairs".  

\begin{Def}
A pair $(\phi,\psi) \in \mathrm{Aut}(V) \times \mathrm{Aut}(L)$ is said to be a \textbf{compatible pair} if the map $\phi: (\rho,D,\theta; V) \to (\rho',D',\theta';V)$ is an automorphism of representations between the induced and the twisted representations of $L$ on $V$. 
\end{Def}

\noindent
We denote by $\C \subseteq \mathrm{Aut}(V) \times \mathrm{Aut}(L)$ the set of all compatible pairs. It immediately follows that $\C$ is a subgroup of $\mathrm{Aut}(V) \times \mathrm{Aut}(L)$. 
Since compatibility is a necessary condition for inducibility, we have 
\[
	\tau\left(\mathrm{Aut}_V(\L)\right) \subseteq \C \subseteq \mathrm{Aut}(V) \times \mathrm{Aut}(L).
\]

\medskip

Let $0 \to V \xrightarrow[]{i} \L \xrightarrow[]{p} L \to 0$ be an abelian extension with a section $s$ of $p$. Our next goal is to find a necessary and sufficient condition for a compatible pair $(\phi,\psi) \in \C$ to be inducible, i.e., 
\[(\phi,\psi) \in \tau(\mathrm{Aut}_V(\L)).\] 
In other words, we want to find when there exists a $\gamma$ such that the following diagram is commutative
\[\begin{tikzcd}
0 & V & \L & L & 0 \\
0 & V & \L & L & 0.
\arrow[from=1-4, to=1-5]
\arrow[from=1-1, to=1-2]
\arrow[from=2-1, to=2-2]
\arrow["i"', from=2-2, to=2-3]
\arrow["p"', from=2-3, to=2-4]
\arrow[from=2-4, to=2-5]
\arrow["\phi"', from=1-2, to=2-2]
\arrow["i", from=1-2, to=1-3]
\arrow["p", from=1-3, to=1-4]
\arrow["\gamma", dashed, from=1-3, to=2-3]
\arrow["\psi", from=1-4, to=2-4]
\end{tikzcd}\]
The answer to this question is related to the $(2,3)$-cohomology group $\mathrm{H}^{(2,3)}(L,V)$. We prove the following lemma before proving the main theorem.

\begin{Lem}
Let $0 \to V \xrightarrow[]{i} \L \xrightarrow[]{p} L \to 0$ be an abelian extension with a section $s$ of $p$. If $(\phi,\psi) \in \C$, then
\[\big(\phi \circ \alpha \circ (\psi^{-1},\psi^{-1}),~ \phi \circ \beta \circ (\psi^{-1}\psi^{-1},\psi^{-1})\big) \in Z^{(2,3)}(L,V),\] 
\noindent
where $(\alpha, \beta)$ is the induced cocycle of the abelian extension ${\L}$ given by equations \eqref{ind_cocycle-I}-\eqref{ind_cocycle-II}.
\end{Lem}

\begin{proof}
From Lemma \ref{Lem_induced-cocycle}, the pair $(\alpha, \beta)$ is a $(2,3)$-cocycle. i.e., for $a,b,c,d,e \in L$, we have 
\begin{equation} \label{alpha_beta cocycle}
\delta_I \alpha(a,b,c,d) = 0,\quad \delta_{II} \beta (a,b,c,d,e) = 0, \quad \delta_I^* \alpha (a,b,c) = 0,\quad \mbox{and}\quad \delta_{II}^* \beta (a,b,c,d) = 0.   
\end{equation}
\noindent
Applying the definition of $\delta$, $\delta^*$ on $\phi \circ \alpha \circ (\psi^{-1},\psi^{-1})$ and $\phi \circ \beta \circ (\psi^{-1},\psi^{-1},\psi^{-1})$, and simplifying we get the following expressions
\begin{align*}
\delta_I\big(\phi \circ \alpha \circ (\psi^{-1},\psi^{-1})\big)(a,b,c,d)
&= 
\phi\big(\delta_I\alpha(\psi^{-1}a,\psi^{-1}b,\psi^{-1}c, \psi^{-1}d)\big), 
\\ 
\delta_{II}\big(\phi \circ \beta \circ (\psi^{-1}, \psi^{-1}, \psi^{-1})\big)(a,b,c,d,e)
&= 
\phi\big(\delta_{II} \beta (\psi^{-1}a, \psi^{-1}b, \psi^{-1}c, \psi^{-1}d, \psi^{-1}e)\big), 
\\ 
\delta_I^*\big(\phi \circ \alpha \circ (\psi^{-1},\psi^{-1}) \big)(a,b,c) 
&= 
\phi \big(\delta_I^*\alpha (\psi^{-1}a,\psi^{-1}b,\psi^{-1}c)\big),
\\ 
\delta_{II}^* \big(\phi \circ \beta \circ (\psi^{-1},\psi^{-1},\psi^{-1})\big)(a,b,c,d) 
&= 
\phi\big( \delta_{II}^* \beta (\psi^{-1}a,\psi^{-1}b,\psi^{-1}c,\psi^{-1}d)\big).
\end{align*}
Since the equation \eqref{alpha_beta cocycle} holds for all $a,b,c,d,e \in L$ and $\psi$ is an automorphism of $L$. It is evident that equation \eqref{alpha_beta cocycle} holds for $\psi^{-1}a, \psi^{-1}b, \psi^{-1}c,\psi^{-1}d, \psi^{-1}e \in L$. Thus,  
\[\big(\phi \circ \alpha \circ (\psi^{-1},\psi^{-1})~,~ \phi \circ \beta \circ (\psi^{-1},\psi^{-1},\psi^{-1})\big) \in Z^{(2,3)}(L,V),\]
\noindent
which completes the proof.
\end{proof}

\noindent
We are now ready to give a necessary and sufficient condition for a compatible pair to be inducible. 

\begin{Thm} \label{Thm_cmptbl_indcbl}
Let $0 \to V \xrightarrow[]{i} \L \xrightarrow[]{p} L \to 0$ be an abelian extension of $L$. A compatible pair $(\phi,\psi)$ is inducible if and only if the $(2,3)$ cocycles  
\[(\alpha,\beta) \quad \text{and} \quad 
\big(\phi \circ \alpha \circ (\psi^{-1},\psi^{-1}),~ \phi \circ \beta \circ (\psi^{-1}\psi^{-1},\psi^{-1})\big)\]
are cohomologous. 
\end{Thm}
\begin{proof}
	($\Longrightarrow$)
	Let $(\phi,\psi) \in \C$ be an inducible pair, then there exists an automoprhism $\gamma \in \mathrm{Aut}_V(\L)$ such that $\tau(\gamma ) = (\phi,\psi)$. Let $s: L \to \L$ be a section of $p$. Then, we have the following commutative diagram 
	\[\begin{tikzcd}
		0 & V & \L & L & 0 \\
		0 & V & \L & L & 0
		\arrow["i", from=1-2, to=1-3]
		\arrow["p", from=1-3, to=1-4]
		\arrow[from=1-4, to=1-5]
		\arrow[from=1-1, to=1-2]
		\arrow[from=2-1, to=2-2]
		\arrow[from=2-4, to=2-5]
		\arrow["i"', from=2-2, to=2-3]
		\arrow["p"', from=2-3, to=2-4]
		\arrow["{\gamma|_V = \phi }"', from=1-2, to=2-2]
		\arrow["\gamma", from=1-3, to=2-3]
		\arrow["{\psi = p \gamma s }", from=1-4, to=2-4].
	\end{tikzcd}\]  
	The goal is to show $(\alpha,\beta)$ and $ 
	\big(\phi \circ \alpha \circ (\psi^{-1},\psi^{-1})
	~,~ \phi \circ \beta \circ (\psi^{-1}\psi^{-1},\psi^{-1})\big)$ are cohomologous. 
	Since $ps = 1_L$, $ p\gamma s (a) = \psi (a) = ps\psi (a) \quad \implies \quad \gamma s (a) - s \psi (a) \in V$, for all $a \in L$. So, we define a map $\lambda: L \to V$ by 
	\[
		\lambda(a):= \gamma s(a) - s \psi(a).
	\]
	For all $a,b \in L$, we have the following expression
	\begin{align*}
		\phi \circ &\alpha \circ (\psi^{-1},\psi^{-1}) (a,b) 
		= \phi(\alpha(\psi^{-1}(a),\psi^{-1}(b))) = \gamma (\alpha(\psi^{-1}(a),\psi^{-1}(b))) \\
		&= \gamma\big([s \psi^{-1}(a),s \psi^{-1}(b)] - s[\psi^{-1}(a),\psi^{-1}(b)]\big) \\
		&= [\gamma s \psi^{-1}(a), \gamma s \psi^{-1}(b)] - \gamma s [\psi^{-1}(a),\psi^{-1}(b)] \\
		&= [\lambda \psi^{-1}(a) + s \psi(\psi^{-1}(a)), \lambda \psi^{-1}(b) + s \psi(\psi^{-1}(b))] \\
		& \quad \quad\quad \quad\quad\quad\quad\quad\quad \quad 
		- \lambda ([\psi^{-1}(a),\psi^{-1}(b)]) - s \psi([\psi^{-1}(a),\psi^{-1}(b)]) \\
		&= [s(a),\lambda \psi^{-1}(b)] - [s(b),\lambda \psi^{-1}(a)] + [s(a),s(b)] - \lambda([\psi^{-1}(a),\psi^{-1}(b)]) - s[a,b]	\\
		&= \delta_I(\lambda \psi^{-1})(a,b) + \alpha(a,b).
	\end{align*}
	Again, for all $a,b,c \in L$, we have
	\begin{align*}
		\phi \circ &\beta \circ (\psi^{-1},\psi^{-1},\psi^{-1}) (a,b,c) 
		= \gamma (\beta(\psi^{-1}(a),\psi^{-1}(b),\psi^{-1}(c))) \\
		&= 
		\gamma\big(\{s \psi^{-1}(a),s \psi^{-1}(b),s \psi^{-1}(c)\} - s\{\psi^{-1}(a),\psi^{-1}(b),\psi^{-1}(c)\}\big) \\
		&= 
		\{\gamma s \psi^{-1}(a), \gamma s \psi^{-1}(b), \gamma s \psi^{-1}(c)\} - \gamma s \{\psi^{-1}(a),\psi^{-1}(b),\psi^{-1}(c)\} \\
		&= 
		\{\lambda \psi^{-1}(a) + s \psi( \psi^{-1}(a)), \lambda \psi^{-1}(b) + s \psi( \psi^{-1}(b)), \lambda \psi^{-1}(c) + s \psi(\psi^{-1}(c))\} \\
		& \quad \quad
		- \lambda(\{\psi^{-1}(a),\psi^{-1}(b),\psi^{-1}(c)\}) - s \psi\{\psi^{-1}(a),\psi^{-1}(b),\psi^{-1}(c)\} \\
		&=
		\{\lambda \psi^{-1}(a),s(b),s(c)\} - \{\lambda \psi^{-1}(b), s(a),s(c)\} + \{s(a),s(b),\lambda \psi^{-1}(c)\} \\
		&\quad \quad
		- \lambda (\{\psi^{-1}(a),\psi^{-1}(b),\psi^{-1}(c)\}) - s(\{a,b,c\}) \\
		&= 
		\delta_{II} (\lambda \psi^{-1})(a,b,c) + \beta(a,b,c)
	\end{align*} 
	Therefore, the cocycles 
	\[
	(\alpha,\beta) ~~\text{and}~~ 
	\big(\phi \circ \alpha \circ (\psi^{-1},\psi^{-1})
	~,~ \phi \circ \beta \circ (\psi^{-1}\psi^{-1},\psi^{-1})\big)
	\]
	are cohomologous.
	
	\medskip
	\noindent
	$(\impliedby)$ Conversely, we assume that the cocycles $(\alpha,\beta)$ and 
	$\big(\phi \circ \alpha \circ (\psi^{-1},\psi^{-1})
	~,~ \phi \circ \beta \circ (\psi^{-1}\psi^{-1},\psi^{-1})\big)$ are cohomologous for $(\phi,\psi) \in \C$. Then, there exists $\lambda \in C^1(L,V)$, i.e, $\lambda: L \to V$  such that 
	\[
		\phi \circ \alpha \circ (\psi^{-1},\psi^{-1}) - \alpha = \delta_I(\lambda), \quad
		\phi \circ \beta \circ (\psi^{-1}\psi^{-1},\psi^{-1}) - \beta =
		\delta_{II}(\lambda).
	\]
This gives us the following expressions	
\begin{align} 
		& \phi \circ \alpha \circ (\psi^{-1},\psi^{-1}) (\psi(a),\psi(b)) - \alpha (\psi(a),\psi(b)) = \delta_I \lambda (\psi(a),\psi(b)) \nonumber \\
		\implies \quad & \phi \circ \alpha (a,b) - \alpha(\psi(a),\psi(b)) = [s \psi(a),\lambda \psi(b)]
		+ [\lambda \psi(a),s \psi(b)] - \lambda([\psi(a),\psi(b)])
	\end{align}
and 
	\begin{align}
		& \phi \circ \beta \circ (\psi^{-1}, \psi^{-1}, \psi^{-1}) (\psi(a),\psi(b),\psi(c)) - \beta (\psi(a),\psi(b), \psi(c))
		= \delta_{II} \lambda(\psi(a),\psi(b),\psi(c)) \nonumber\\
		\implies \quad & \phi \circ \beta(a,b,c) -\beta(\psi(a),\psi(b),\psi(c)) 
		 = \{\lambda \psi(a), s\psi(b), s\psi(c)\} - \{\lambda \psi(b), s \psi(a), s \psi(c)\}  \nonumber \\
		&~~ \quad\quad\quad\quad\quad\quad\quad\quad\quad\quad\quad\quad\quad\quad\quad\quad\quad \quad \quad 
		+ \{s \psi(a),s \psi(b), \lambda \psi(c)\}- \lambda(\psi(a),\psi(b),\psi(c)),  \label{equ_2}
	\end{align}
for all $a,b,c\in L$. Let $s:L \to \L$ be a section of $p$. we know that as a vector space $\L \cong V \oplus s(L)$. So, any element $\tilde{l} \in \L$ can be uniquely represented as $\tilde{l} = v + s(a)$ for some $v \in V$ and $a \in L$. We now want to show $(\phi,\psi) \in \C$ is inducible. So we define a map $\gamma: \L \to \L$ as follows: 
	\[
		\gamma(v + s(a)) := \phi(v) + \lambda \psi(a) + s \psi(a).
	\]   
	It is clear from the definition of $\gamma$ that $\tau(\gamma) = (\phi,\psi)$, i.e., $\gamma$ induces the pair $(\phi,\psi)$.
	
	\vspace{0.1in}
	Now, it only remains to check $\gamma \in \mathrm{Aut}_V(\L)$. Injectivity of $\gamma$ follows from the injectivity of $s, \phi, \psi$. To see $\gamma$ is surjective, we take an arbitrary element of $v + s(a) \in \L$. Then 
	\[
		\gamma\big(\phi^{-1}(v - \lambda(a)) + s\psi^{-1}(a)\big) = v+s(a).
	\] 
	So, the map $\gamma$ is surjective. We also need to  show that $\gamma$ is a Lie-Yamaguti algebra morphism, i.e., $\gamma$ preserves the brackets. For that let $\tilde{l}_1 = v_1 + s(a_1)$ and $\tilde{l}_2 = v_2 + s(a_2)$ be two elements in $L$, then 
	\begin{align*}
		\gamma[\tilde{l}_1,\tilde{l}_2] 
		&= \gamma([v_1 + s(a_1),v_2+s(a_2)])
		= \gamma([v_1,s(a_2)]) + \gamma([s(a_1),v_2]) + \gamma([s(a_1),s(a_2)]) \\
		&= \phi([s(a_1),v_2]) - \phi([s(a_2),v_1]) + \gamma(\alpha(a_1,a_2) + s[a_1,a_2]) \\
		&= 
		\phi \rho(a_1) (v_2) - \phi \rho(a_2) (v_1) + \phi \alpha(a_1,a_2) + \gamma s[a_1,a_2].
	\end{align*}
	Moreover,   
	\begin{align*}
		[\gamma(\tilde{l}_1),\gamma(\tilde{l}_2)] 
		&= [\gamma(v_1+s(a_1)),\gamma(v_2+s(a_2))] \\
		&= [\phi(v_1)+\lambda \psi(a_1) + s \psi(a_1)~,~\phi(v_2)+\lambda \psi(a_2) + s \psi(a_2)] \\
		&= [\phi(v_1),s \psi(a_2)] + [\lambda \psi(a_1), s\psi(a_2)] 
		+ [s \psi(a_1),\phi(v_2)] 
		+ [s \psi(a_1),\lambda \psi(a_2)] \\
		&\quad \quad
		+ [s \psi(a_1),s \psi(a_2)]\\
		&= 
		\rho(\psi(a_1))(\phi(v_2)) - \rho(\psi(a_2))(\phi(v_1))
		+ \phi \alpha (a_1,a_2) - \alpha(\psi(a_1),\psi(a_2)) \\
		& \quad \quad
		+ \lambda([\psi(a_1),\psi(a_2)]) + \alpha(\psi(a_1),\psi(a_2)) + s[\psi(a_1),\psi(a_2)] \\
		&=
		\rho(\psi(a_1))(\phi(v_2)) - \rho(\psi(a_2))(\phi(v_1))
		+ \phi \alpha (a_1,a_2) + \lambda \psi[a_1,a_2] + s \psi[a_1,a_2] \\
		&=
		\rho(\psi(a_1))(\phi(v_2)) - \rho(\psi(a_2))(\phi(v_1))
		+ \phi \alpha (a_1,a_2) + \gamma s[a_1,a_2].
	\end{align*}
	From the compatibility condition of $(\phi,\psi)$, we have 
	\[
		\gamma[\tilde{l}_1,\tilde{l}_2] = [\gamma(\tilde{l}_1),\gamma(\tilde{l}_2)].
	\]
	Similarly, $\gamma\{\tilde{l}_1,\tilde{l}_2,\tilde{l}_3\} = \{\gamma(\tilde{l}_1),\gamma(\tilde{l}_2),\gamma(\tilde{l}_3)\}$ follows from equation \eqref{equ_2}. 
	Hence, $\gamma \in \mathrm{Aut}_V(\L)$ with $\tau(\gamma) = (\phi,\psi)$, which shows $(\phi,\psi)$ is an inducible pair. 
\end{proof}

\bigskip

\section{\large Wells exact sequence and inducibility problem}\label{sec-5}

\medskip

This section introduces the Wells map associated with an abelian extension of Lie-Yamaguti algebras. Then we develop an analog of the classical Wells exact sequence known for group extensions (see \cite{jin,passi,robinson,wells}) in the context of Lie-Yamaguti algebras. In the sequel, we discuss the inducibility problem for Lie-Yamaguti algebra extensions in terms of the Wells map. Finally, we consider two other relevant questions related to the inducibility problem. As an application to these questions, we describe certain automorphism groups of semi-direct Lie-Yamaguti algebras.

\medskip

Let $0 \to V \xrightarrow[]{i} \L \xrightarrow[]{p} L \to 0$ be an abelian extension of $L$ by $V$, with $s$ being a section of $p$. Let $(\rho,D,\theta;V)$ be the induced representation of $L$ on $V$, and $(\alpha,\beta)$ be the induced cocycle. For any $(\phi,\psi) \in \mathrm{Aut}(V) \times \mathrm{Aut}(L)$, we define $\alpha_{\phi,\psi} \in C^2(L,V)$ and $\beta_{\phi,\psi} \in C^3(L,V)$ by 
\[\alpha_{\phi,\psi}(a,b) := \phi \alpha(\psi^{-1}(a),\psi^{-1}(b)) - \alpha(a,b),\]  
\[\beta_{\phi,\psi}(a,b,c) := \phi \beta(\psi^{-1}(a),\psi^{-1}(b),\psi^{-1}(c)) - \beta(a,b,c).\]

\medskip

\noindent
Now, we define a set map $\mathcal{W}: \mathrm{Aut}(V) \times \mathrm{Aut}(L) \to \mathrm{H}^{(2,3)}(L,V)$ by the following formula
\[\mathcal{W}((\phi,\psi)) = [(\alpha_{\phi,\psi},\beta_{\phi,\psi})],\]
the cohomology class of $(\alpha_{\phi,\psi},\beta_{\phi,\psi})$. The map $\mathcal{W}$ is called the \textbf{Wells map}. In general, $\mathcal{W}$ may not be a group homomorphism. From Lemma \ref{Lem_ind of sec},  it is evident that the Wells map $\mathcal{W}$ does not depend on the choice of a section of $p$.

\begin{Rem}
It follows from Theorem \ref{Thm_cmptbl_indcbl} that a compatible pair $(\phi,\psi) \in \mathcal{C}$ is inducible if and only if $\mathcal{W}((\phi,\psi))$ is trivial. Thus, $\mathcal{W}((\phi,\psi))$ can be seen as an obstruction for inducibility of the pair $(\phi,\psi)$.   
\end{Rem}

\noindent Let us introduce a few notations that will be used in the rest of the paper.
\[\mathrm{Aut}_V^L(\L) := \{\gamma \in \mathrm{Aut}_V(\L): \bar{\gamma} = id_L\}, \quad
\mathrm{Aut}^{V,L} (\L) := \{\gamma \in \mathrm{Aut}_V(\L): \gamma|_V = id_V,~ \bar{\gamma} = id_L\},
\] 
\[
\mathrm{Aut}^V(\L) := \{\gamma \in \mathrm{Aut}_V(\L): \gamma|_V\ = id_V\}.
\] 

\smallskip

We now develop the Wells exact sequence for Lie-Yamaguti algebra extensions. First, we show that the first cohomology group $\mathrm{H}^1(L,V)$ 
is isomorphic to the automorphism group $\mathrm{Aut}^{V,L}(\L)$.

\begin{Lem} \label{Lem_Wells_Iso}
Let $0 \to V \xrightarrow[]{i} \L \xrightarrow[]{p} L \to 0$ be an abelian extension of $L$ by $V$. Then, we have the following isomorphism
\[\mathrm{Aut}^{V,L}(\L) \cong \mathrm{H}^1(L,V).\] 	
\end{Lem}
\begin{proof}
Let us fix a section $s$ of $p$, i.e., a linear map $s:L\rightarrow \L$. For any $\gamma \in \mathrm{Aut}^{V,L}(\L)$, we define a map $\lambda_{\gamma} : L \to V$ as follows 
\[\lambda_{\gamma}(a) = \gamma s (a) - s(a).\]
Since $\gamma \in \mathrm{Aut}^{V,L}(\L)$, we have $\gamma (\tilde{l}) - \tilde{l} \in V$  for all $\tilde{l} \in \L$. Thus, $\lambda_{\gamma}$ is well defined. Now, we show that $\lambda_{\gamma} \in Z^1(L,V)$. For $a,b \in L$,
\begin{align*}
\lambda_{\gamma}[a,b] &= \gamma s[a,b] - s[a,b] 
= \gamma\big([s(a),s(b)] - \alpha(a,b)\big) + \alpha(a,b) - [s(a),s(b)] \\
&= [\gamma s(a),\gamma s(b)] - [s(a),s(b)] 
= [\lambda_{\gamma}(a) + s(a),\lambda_{\gamma}(b)+s(b)] - [s(a),s(b)] \\
&= [\lambda_{\gamma}(a),s(b)] + [s(a),\lambda_{\gamma}(b)].
\end{align*}  
From this gives us the following expression, 
\[\delta_I (\lambda_{\gamma})(a,b) = [s(a),\lambda_{\gamma}(b)] - [s(b),\lambda_{\gamma}(a)] - \lambda_{\gamma}[a,b] = 0.\]
By a similar computation one can show $\delta_{II} (\lambda_{\gamma})(a,b,c) = 0$. Hence,  $\lambda_{\gamma} \in \mathrm{H}^1(L,V)$. This leads us to define a function 
\[\chi: \mathrm{Aut}^{V,L}(\L) \to \mathrm{H}^1(L,V) , \quad  \chi(\gamma) := \lambda_{\gamma}.\]
We now show that the map $\chi$ is a group isomorphism.

\smallskip

\noindent 	
\textit{$\chi$ is a group homomorphism:}
Let $\gamma_1,~ \gamma_2 \in \mathrm{Aut}^{V,L}(\L)$. Then, for any $a \in L$, we have
\begin{align*}
\lambda_{\gamma_1\circ \gamma_2}(a) 
&= \gamma_1 \circ \gamma_2 s (a) - s(a) = \gamma_1 \circ \gamma_2 s(a) - \gamma_1 s(a) + \gamma_1 s(a) - s(a) \\
&= \gamma_1(\gamma_2 s(a) - s(a)) + \gamma_1 s(a) - s(a) \\
&= \gamma_1(\lambda_{\gamma_2}(a)) + \lambda_{\gamma_1} (a)
= \lambda_{\gamma_2}(a) + \lambda_{\gamma_1}(a).
\end{align*} 
Note that $\gamma_1(\lambda_{\gamma_2}(a)) = \lambda_{\gamma_2}(a)$ since $\gamma_1 \in \mathrm{Aut}^{V,L}(\L)$ and $\lambda_{\gamma_2}(a) \in V$. Therefore, $\lambda_{\gamma_1\circ \gamma_2} = \lambda_{\gamma_1} + \lambda_{\gamma_2}$, i.e., $\chi$ is a group homomorphism. 
	
\smallskip
\noindent
\textit{$\chi$ is an injective map:}
Let $\gamma \in \mathrm{Aut}^{V,L}(\L)$ with $\chi(\gamma) = 0$. Then for all $a \in L$ and $v \in V$, 
\[
\gamma s(a) = s(a) \quad\text{and}\quad \gamma(v) = v.
\]
Since $\L \cong V \oplus s(L)$ as a vector space, any element $\tilde{l} \in \L$ can be uniquely expressed in the form $\tilde{l} = v + s(a)$. Then, $\gamma(\tilde{l}) = \gamma(v + s(a)) = v + s(a) = \tilde{l}$, which implies that $\gamma = 1_{\L}$. So $\mathrm{Ker}(\chi) = 1_{\L}$. Therefore, the map $\chi$ is injective. 
	
\smallskip
\noindent
\textit{$\chi$ is a surjective map:}
Let $\lambda \in \mathrm{H}^1(L,V)$. We define a map  
\[
\gamma: \L \to \L, \quad \gamma(v + s(a)) := v + \lambda(a) + s(a).
\]   
It is clear from the definition that $\chi(\gamma) = \lambda$. We only need to check that $\gamma \in \mathrm{Aut}^{V,L}(\L)$. From the definition of $\gamma$ one has $\gamma|_V = 1_V$ and $p \gamma s = 1_L$. Now to show the injectivity of the map $\gamma$, let us consider the following expression
\[
\gamma(v+s(a)) = 0 \quad \implies \quad s(a) = 0 \quad \implies \quad a = 0, 
\]
which, in turn, implies that $v = 0$. So $\gamma$ is injective. Also, 
\[
\gamma(v - \lambda(a) + s(a)) = v + s(a).
	\]
	implies $\gamma$ is surjective. Thus, $\gamma: \L \to \L$ is a bijective map. It remains to show that the map $\gamma$ preserves the two brackets. First, we recall the definition of $\alpha$ and $\beta$. For $a,b,c \in L$,
	\[
	\alpha(a,b) = [s(a),s(b)] - s[a,b], \quad \mbox{and} \quad\beta(a,b,c) = \{s(a),s(b),s(c)\} - s\{a,b,c\}.
	\]
Let $\tilde{l}_1,\tilde{l}_2 \in \L$, then we know that $\tilde{l}_1 = v_1 + s(a_1)$, $\tilde{l}_2 = v_2 + s(a_2)$. Let us consider the following expression 
	\begin{align*}
		\gamma[\tilde{l}_1,\tilde{l}_2] 
		&= \gamma[v_1 + s(a_1),v_2+s(a_2)] 
		= \gamma\big([v_1,s(a_2)] + [s(a_1),v_2] + [s(a_1),s(a_2)]\big) \\
		&= [v_1,s(a_2)] + [s(a_1),v_2] + \gamma[s(a_1),s(a_2)] ~~~(\text{as}~ \gamma|_V =1_V) \\
		&= [v_1,s(a_2)] + [s(a_1),v_2] + \alpha(a_1,a_2) + \gamma s[a_1,a_2] \\
		&= [v_1,s(a_2)] + [s(a_1),v_2] + \alpha(a_1,a_2) + \lambda[a_1,a_2] + s[a_1,a_2] \\
		&= [v_1,s(a_2)] + [s(a_1),v_2] + \alpha(a_1,a_2) + 
		[\lambda(a_1),s(a_2)] + [s(a_1),\lambda(a_2)] + s[a_1,a_2] \\
		&= [v_1,s(a_2)] + [s(a_1),v_2] + [s(a_1),s(a_2)] + 
		[\lambda(a_1),s(a_2)] + [s(a_1),\lambda(a_2)] \\
		&= [\gamma(\tilde{l}_1), \gamma(\tilde{l}_2)].
	\end{align*} 
	By a similar calculation, it follows that $\gamma\{\tilde{l}_1,\tilde{l}_2,\tilde{l}_3\} = \{\gamma(\tilde{l}_1),\gamma(\tilde{l}_2),\gamma(\tilde{l}_3)\}$. Thus, $\gamma \in \mathrm{Aut}^{V,L}(\L)$.  
\end{proof}

\medskip
With these notations, we construct the Wells exact sequence for Lie-Yamaguti algebra extensions in the next result. 

\begin{Thm} \label{Thm_Wells_Sq}
Let $0 \to V \xrightarrow[]{i} \L \xrightarrow[]{p} L \to 0$ be an abelian extension of $L$ by $V$, and $\mathcal{W}$ be the associated Wells map. Then, there is an exact sequence 
\[\begin{tikzcd}
0 & {\mathrm{H}^1(L,V)} & {\mathrm{Aut}_V(\L)} & \mathcal{C} & {\mathrm{H}^{(2,3)}(L,V)}
\arrow[from=1-1, to=1-2]
\arrow["i", from=1-2, to=1-3]
\arrow["\tau", from=1-3, to=1-4]
\arrow["\mathcal{W}", from=1-4, to=1-5].
\end{tikzcd}\]
This sequence is known as \textbf{Wells exact sequence}. 
\end{Thm}

\begin{proof}
The map $i$ is the inclusion map via the isomorphism of Lemma \ref{Lem_Wells_Iso}, so exactness at the first term is trivial.
	
\smallskip
\noindent
\textit{Exactness at the second term:} If $\gamma \in \mathrm{Ker}(\tau)$, then $\tau(\gamma) = (1_V,1_L)$. Thus, $\gamma \in \mathrm{Aut}^{V,L}(\L)\cong \mathrm{H}^1(L,V)$, which implies that $Ker(\tau) \subseteq \mathrm{H}^1(L,V)$. Conversely, let $\gamma \in \mathrm{H}^1(L,V) \cong \mathrm{Aut}^{V,L}(\L)$, then it follows that $\tau(\gamma) = (1_V,1_L)$. Therefore, $\mathrm{H}^1(L, V) \subseteq Ker(\tau)$ and we have the exactness at the second term.   
	
\smallskip
\noindent
\textit{Exactness at the third term:} 
If $(\phi,\psi) \in Ker(\mathcal{W})$, then $(\phi,\psi)$ is a $(2,3)$-coboundary. Theorem \ref{Thm_cmptbl_indcbl} shows that the pair is inducible. Let $\tau(\gamma) = (\phi,\psi)$ for some $\gamma \in \mathrm{Aut}_V(\L)$. Thus, $Ker(\mathcal{W}) \subseteq  Im(\tau)$. Conversely, if $(\phi,\psi) \in Im(\tau)$, then $(\phi,\psi)$ is inducible. From Theorem \ref{Thm_cmptbl_indcbl}, $(\alpha_{\phi,\psi},\beta_{\phi,\psi})$ is a coboundary, which implies that $\mathcal{W}(\phi,\psi) = 0$. So, $Im(\tau) \subseteq Ker(\mathcal{W})$. Subsequently, the sequence is exact at the third term.    
\end{proof}

\medskip

\begin{Cor} \label{Comp_iff_Ind}
If $0 \to V \xrightarrow[]{i} \L \xrightarrow[]{p} L \to 0$ is a split extension of Lie-Yamaguti algebras, then every compatible pair $(\phi,\psi) \in \mathcal{C}$ is inducible. 
\end{Cor}
\begin{proof}
An extension $0 \to V \xrightarrow[]{i} \L \xrightarrow[]{p} L \to 0$ is called a split extension if there is a section $s$ of $p$ which is a Lie-Yamaguti algebra morphism. Hence both $\alpha_{\phi,\psi}$ and $\beta_{\phi,\psi}$ are zero, giving 
\[\mathrm{Ker}(\mathcal{W}) = \mathcal{C}.\] 
The Wells exact sequence implies that $\mathrm{Im}(\tau)=\mathrm{Ker}(\mathcal{W}) = \mathcal{C}$. 
\end{proof}

\bigskip

\noindent 
\textbf{{Some other questions related to the inducibility problem:}}  

\smallskip
\noindent
Let ${\L}$ be an abelian extension of a Lie-Yamaguti algebra $L$ by $V$. Then, we address the following questions about lifting automorphisms in an abelian extension.   

\medskip
\noindent
\textbf{Question I.} When is a pair $(\phi,1) \in \C$ inducible? In other words, can we lift an automorphism of $L$ to an automorphism of $\L$ fixing $V$ point-wise?
	
\medskip
\noindent	
\textbf{Question II.} When is a pair $(1, \psi) \in \C$ inducible? In other words, can we extend an automorphism of $V$ to an automorphism of $\L$ such that the extension induces the identity map on $L$, i.e., $p \gamma s = 1_L$?

\bigskip

\noindent
Generally, for any $\phi \in \mathrm{Aut}(V)$ (respectively $\psi \in \mathrm{Aut}(L)$) we do not have $(\phi,1) \in \C$ (respectively $(1,\psi) \in \C$). So, we define the following collections of compatible pairs
\[
\mathcal{C}_1 :=\{\phi \in \mathrm{Aut}(V): (\phi,1) \in \C\}
\quad \mbox{and}\quad
\mathcal{C}_2 := \{\psi \in \mathrm{Aut}(L): (1,\psi) \in \C\}.
\]
The sets $\C_1$ and $\C_2$ are subgroups of $\mathrm{Aut}(V)$ and $\mathrm{Aut}(L)$, respectively. Now, to answer Question I and Question II, we consider the following sequences 
\[\begin{tikzcd}
	& 0 & {\mathrm{Aut}^{V,L}(\L)} & {\mathrm{Aut}_V^L(\L)} & {\mathcal{C}_1} & {\mathrm{H}^{(2,3)}(L,V)}, & (A)\\
	& 0 & {\mathrm{Aut}^{V,L}(\L)} & {\mathrm{Aut}^V(\L)} & {\mathcal{C}_2} & {\mathrm{H}^{(2,3)}(L,V)}. & (B)\\
	\arrow[from=1-2, to=1-3]	
	\arrow["i_1", from=1-3, to=1-4]
	\arrow["{\tau_1}", from=1-4, to=1-5]
	\arrow["{\lambda_1}", from=1-5, to=1-6]
	\arrow[from=2-2, to=2-3]
	\arrow["i_2", from=2-3, to=2-4]
	\arrow["{\tau_2}", from=2-4, to=2-5]
	\arrow["{\lambda_2}", from=2-5, to=2-6]
\end{tikzcd}
\vspace{-0.3in}\]
where $i_1, i_2$ are inclusions maps and $\tau_1, \tau_2, \lambda_1$, and $\lambda_2$ are defined as follows.

\medskip
\noindent \textbf{1. Definition of the maps $\tau_1$ and $\tau_2$:} 

\smallskip

\noindent For $\gamma \in \mathrm{Aut}_V^L(\L)$ such that $\tau(\gamma) = (\phi,1)$, we define $\tau_1(\gamma):= \phi$. For $\gamma \in \mathrm{Aut}^V(\L)$ such that $\tau(\gamma) = (1,\psi)$, we define $\tau_2(\gamma) := \psi$. These definitions make sense because Lemma \ref{Lem_compatibility condition} shows that $(\phi,1)$ and $(1,\psi)$ are compatible pairs, i.e., $\phi\in \C_1$ and $\psi
\in \C_2$.

\medskip

\noindent  \textbf{2. Definition of the maps $\lambda_1$ and $\lambda_2$:}
\[
	\lambda_1(\phi) := \big[(\phi \alpha - \alpha,~ \phi \beta - \beta)\big], 
	\quad
	\lambda_2(\psi) := \big[\big(\alpha \circ (\psi^{-1},\psi^{-1}) - \alpha,~ \beta \circ (\psi^{-1},\psi^{-1},\psi^{-1}) - \beta\big)\big], 
\]  
where $[\cdot,\cdot]$ denotes the cohomology class in $\mathrm{H}^{(2,3)}(L,V)$. The above definition makes sense only if the following conditions hold.

\smallskip
 
\begin{itemize}
	\item $(\phi \alpha - \alpha ~,~ \phi \beta - \beta)$ and $(\alpha \circ (\psi^{-1},\psi^{-1}) - \alpha ~,~ \beta \circ (\psi^{-1},\psi^{-1},\psi^{-1}) - \beta)$ should be $(2,3)$-cocycles. 
	
	\smallskip
		
	\item Definition of $\lambda_1$ and $\lambda_2$ do not depend on the choice of the section $s$.
	\end{itemize} 

\medskip

\noindent
The following lemmas ensure that $\lambda_1$ and $\lambda_2$ are well-defined. 

\begin{Lem}
For any $\phi \in \C_1$ and $\psi \in \C_2$, the following $(2,3)$-cochains are $(2,3)$-cocycles.
\[\big(\phi \alpha - \alpha ~,~ \phi \beta - \beta\big) \quad\mbox{and}\quad \big(\alpha \circ (\psi^{-1},\psi^{-1}) - \alpha ~,~ \beta \circ (\psi^{-1},\psi^{-1},\psi^{-1}) - \beta\big)\] 
\end{Lem}

\begin{proof}

Let us recall from lemma (\ref{Lem_induced-cocycle}) that $(\alpha, \beta)$ is a $(2,3)$-cocycle. To show that $\big(\phi \alpha - \alpha ~,~ \phi \beta - \beta\big)$ is a $(2,3)$-cocycle, we only need to show that for $\phi \in \C_1$, $(\phi \alpha, \phi \beta)$ is a $(2,3)$-cocycle. Lemma \ref{Lem_induced-cocycle} and Lemma \ref{Lem_compatibility condition} imply that 
\begin{align*}
		\delta_I \phi \alpha(a_1,&a_2,a_3,a_4) \\
		&=
		- \rho(a_3)\phi \beta(a_1,a_2,a_4) 
		+ \rho(a_4)\phi \beta(a_1,a_2,a_3) 
		+ \phi \beta(a_1,a_2,[a_3,a_4]) \\
		&\qquad
		+ D(a_1,a_2) \phi \alpha(a_3,a_4)
		- \phi \alpha(\{a_1,a_2,a_3\},a_4) 
		- \phi \alpha(a_3,\{a_1,a_2,a_4\}) \\
		&= \phi \delta_I\alpha = 0
\end{align*}
\begin{align*} 
		\delta_{II}\phi \beta(a_1,&a_2,a_3,a_4,a_5) \\
		&=
		- \theta(a_4,a_5)\phi \beta(a_1,a_2,a_3) 
		+ \theta(a_3,a_5)\phi \beta(a_1,a_2,a_4)
		+ D(a_1,a_2)\phi \beta(a_3,a_4,a_5) \\
		&\qquad
		- D(a_3,a_4)\phi \beta(a_1,a_2,a_5) 
		- \phi \beta(\{a_1,a_2,a_3\},a_4,a_5) 
		- \phi \beta(a_3,\{a_1,a_2,a_4\},a_5) \\
		&\qquad
		- \phi \beta(a_3,a_4,\{a_1,a_2,a_5\})  
		+ \phi \beta(a_1,a_2,\{a_3,a_4,a_5\}) \\
		&= \phi \delta_{II} \beta = 0 
		\\ ~ \\
		\delta_I^*\phi \alpha(a_1,&a_2,a_3)\\
		&= 
		- \sum_{\circlearrowleft(a_1,a_2,a_3)}\rho(a_1)\phi \alpha(a_2,a_3) 
		~+ \sum_{\circlearrowleft(a_1,a_2,a_3)}\phi\alpha([a_1,a_2],a_3)
		~+ \sum_{\circlearrowleft(a_1,a_2,a_3)} \phi\beta(a_1,a_2,a_3) \\
		&=
		\phi \delta_I^* \alpha = 0 
		\\ ~ \\
		\delta_{II}^* \phi \beta(a_1,&a_2,a_3,a_4) \\
		&= 
		\theta(a_1,a_4)\phi \alpha(a_2,a_3) + \theta(a_2,a_4)\phi \alpha(a_3,a_4) + \theta(a_3,a_4) \phi \alpha(a_1,a_2) \\
		&\qquad
		+ \beta([a_1,a_2],a_3,a_4) + \beta([a_2,a_3],a_1,a_4) + \beta([a_3,a_1],a_2,a_4) \\	
		&= \phi \delta_{II}^* \beta(a_1,a_2,a_3,a_4) = 0 
	\end{align*}
	Thus, the pair  $(\phi \alpha, \phi \beta)$ is a $(2,3)$-cocycle. A similar calculation implies that the pair $\big(\alpha \circ (\psi^{-1},\psi^{-1}) - \alpha ~,~ \beta \circ (\psi^{-1},\psi^{-1},\psi^{-1}) - \beta\big)$ is a $(2,3)$ cocycle.
\end{proof}

\begin{Lem}
	The maps $\lambda_1$ and $\lambda_2$ do not depend on the choice of $s$. 
\end{Lem}
\begin{proof}
	The maps $\lambda_1$ and $\lambda_2$ are defined by
	\[
		\lambda_1(\phi) := \big[\phi \alpha - \alpha ~,~ \phi \beta - \beta\big] 
		\quad \mbox{and}\quad
		\lambda_2(\psi) := \big[\alpha \circ (\psi^{-1},\psi^{-1}) - \alpha ~,~ \beta \circ (\psi^{-1},\psi^{-1},\psi^{-1}) - \beta\big]. 
	\]  
	Let $s,t$ be two sections of $p$ and their corresponding induced cocycles be $(\alpha_s,\beta_s)$ and $(\alpha_t,\beta_t)$. We want to show that 
	\begin{equation} \label{equ_15}
		\big[\phi \alpha_s - \alpha_s ~,~ \phi \beta_s - \beta_s\big] 
		= 
		\big[\phi \alpha_t - \alpha_t ~,~ \phi \beta_t - \beta_t\big],
	\end{equation}
	\begin{multline} \label{equ_16}
		\big[\alpha_s \circ (\psi^{-1},\psi^{-1}) - \alpha_s ~,~ \beta_s \circ (\psi^{-1},\psi^{-1},\psi^{-1}) - \beta_s\big]
		\\
		=
		\big[\alpha_t \circ (\psi^{-1},\psi^{-1}) - \alpha_t ~,~ \beta_t \circ (\psi^{-1},\psi^{-1},\psi^{-1}) - \beta_t\big].
	\end{multline}
	From Lemma \ref{Lem_ind of sec}, it follows that 
	\begin{equation} \label{equ_17}
		[\alpha_s,\beta_s] = [\alpha_t,\beta_t].
	\end{equation} 
	Once we show that 
	\[
		\big[\alpha_s \circ (\psi^{-1},\psi^{-1})~,~\beta_s \circ (\psi^{-1},\psi^{-1},\psi^{-1})\big]
		=
		\big[\alpha_t \circ (\psi^{-1},\psi^{-1})~,~\beta_t \circ (\psi^{-1},\psi^{-1},\psi^{-1})\big],
	\]
	the equation (\ref{equ_16}) holds. For all $a,b,c \in L$,  
	\begin{align*}
		\beta_s \circ &(\psi^{-1},\psi^{-1},\psi^{-1}) (a,b,c)
		- 
		\beta_t \circ (\psi^{-1},\psi^{-1},\psi^{-1}) (a,b,c) \\
		&= 
		\theta(\psi^{-1}b,\psi^{-1}c)(\lambda \psi^{-1}a) 
		- \theta(\psi^{-1}a,\psi^{-1}c)(\lambda \psi^{-1}b)
		+ D(\psi^{-1}a,\psi^{-1}b)(\lambda \psi^{-1}c)
		- \lambda\psi^{-1}(a,b,c) \\
		&=
		\theta(b,c)(\lambda \psi^{-1}a) 
		- \theta(a,c)(\lambda \psi^{-1}b)
		+ D(a,b)(\lambda \psi^{-1}c)
		- \lambda\psi^{-1}(a,b,c) \quad\quad\quad (\text{by Lemma} ~\ref{Lem_compatibility condition})\\
		&= 
		~~\delta_{II}(\lambda \psi^{-1})(a,b,c).
	\end{align*}
	Similarly, $$\alpha_s \circ (\psi^{-1},\psi^{-1})(a,b) - \alpha_t \circ (\psi^{-1},\psi^{-1})(a,b) = \delta_I(\lambda \psi^{-1})(a,b).$$ 
 Thus, 
	\[
		\big[\alpha_s \circ (\psi^{-1},\psi^{-1})~,~\beta_s \circ (\psi^{-1},\psi^{-1},\psi^{-1})\big]
		=
		\big[\alpha_t \circ (\psi^{-1},\psi^{-1})~,~\beta_t \circ (\psi^{-1},\psi^{-1},\psi^{-1})\big].
	\]
	Now, by using equation (\ref{equ_17}) the equation (\ref{equ_16}) follows. We obtain (\ref{equ_15}) by a similar calculation.  
\end{proof}

\medskip

The next result shows that $(A)$ and $(B)$ are exact sequences, which would answer Questions I and II in the following way

\begin{itemize}
\item a pair $(\phi,1) \in \C$ is inducible if and only if $\phi \in \mathrm{Ker}(\lambda_1)$,
\smallskip
\item a pair $(1,\psi) \in \C$ is inducible if and only if $\psi \in \mathrm{Ker}(\lambda_2)$.
\end{itemize}

\begin{Thm} \label{Thm_two_exact_seq}
Let $0 \to V \xrightarrow[]{i} \L \xrightarrow[]{p} L \to 0$ be an abelian extension of $L$ by $V$. The following sequences are exact sequences 
\[\begin{tikzcd}
0 & {\mathrm{Aut}^{V,L}(\L)} & {\mathrm{Aut}_V^L(\L)} & {\mathcal{C}_1} & {\mathrm{H}^{(2,3)}(L,V)}, \\
0 & {\mathrm{Aut}^{V,L}(\L)} & {\mathrm{Aut}^V(\L)} & {\mathcal{C}_2} & {\mathrm{H}^{(2,3)}(L,V)}.
\arrow[from=1-1, to=1-2]
\arrow["i_1", from=1-2, to=1-3]
\arrow["{\tau_1}", from=1-3, to=1-4]
\arrow["{\lambda_1}", from=1-4, to=1-5]
\arrow[from=2-1, to=2-2]
\arrow["i_2", from=2-2, to=2-3]
\arrow["{\tau_2}", from=2-3, to=2-4]
\arrow["{\lambda_2}", from=2-4, to=2-5]
\end{tikzcd}\]
All the maps are group homomorphisms except $\lambda_1$ and $\lambda_2$.
\end{Thm}
\begin{proof}
The map $i_1$ and $i_2$ are inclusion maps, so exactness at the first term of $(A)$ and $(B)$ is trivial. From the definition of $i_1, i_2, \tau_1$, and $\tau_2$ we have 
\[Im(i_1) = Ker(\tau_1)\quad \text{and} \quad Im(i_2) = Ker(\tau_2)\]
giving us exactness at the second term of $(A)$ and $(B)$.

\medskip

\noindent
\textit{Exactness at the third term of $(A)$:} We want to show that $Im(\tau_1) = Ker(\lambda_1)$. Let $\phi \in Im(\tau_1) \subseteq \C_1$ such that $\gamma \in \mathrm{Aut}^L_V(\L)$ with $\tau_1(\gamma) = \phi$, then $(\phi, 1)$ is a inducible compatible pair. Theorem \ref{Thm_cmptbl_indcbl} implies that $(\phi \alpha - \alpha, \phi \beta - \beta) \in B^2(L,V) \times B^3(L,V)$ since $\psi =1$. Thus, $\phi \in Ker(\lambda_1)$ and we have $Im(\tau_1) \subseteq Ker(\lambda_1)$. 

\smallskip
	
Conversely, if $\phi \in Ker(\lambda_1) \subseteq \C_1$, then $(\phi,1)$ is a compatible pair with $\lambda_1(\phi) = 0$. So, the $(2,3)$-cocycles $(\phi \alpha, \phi \beta)$ and $(\alpha,\beta)$ are cohomologous. Theorem \ref{Thm_cmptbl_indcbl} implies that $(\phi,1)$ is inducible, i.e., there exists a $\gamma \in \mathrm{Aut}_V(\L)$ such that $\tau(\gamma) = (\phi,1)$.  Clearly, $\gamma \in \mathrm{Aut}_V^L(\L)$ with $\tau_1(\gamma) = \phi$. So, we get $\phi \in Im(\tau_1)$, which implies that 
$Im(\tau_1) = Ker(\lambda_1).$    

\medskip

\noindent
\textit{Exactness at the third term of $\mathrm{(B)}$:} We want to show that $Im(\tau_2) = Ker(\lambda_2)$. If $\psi \in Im(\tau_2) \subseteq \C_2$, then we have $\tau_2(\gamma) = \psi$ for some $\gamma \in \mathrm{Aut}^V(\L)$. So, the pair $(1,\psi)$ is inducible. Theorem \ref{Thm_cmptbl_indcbl} implies that $\big(\alpha(\psi^{-1},\psi^{-1}), \beta(\psi^{-1},\psi^{-1},\psi^{-1})\big)$ and $(\alpha,\beta)$ are cohomologous since $\phi = 1$.  Thus, $\psi \in Ker(\lambda_2)$ and it follows that $Im(\tau_2) \subseteq Ker(\lambda_2)$.
	
\smallskip

Conversely, let $\psi \in Ker(\lambda_2) \subseteq \C_2$. Then $(1,\psi)$ is a compatible pair with $\lambda_2(\psi) = 0$. So, the $(2,3)$-cocycles  $\big(\alpha(\psi^{-1},\psi^{-1}), \beta(\psi^{-1},\psi^{-1},\psi^{-1})\big)$ and $(\alpha,\beta)$ are cohomologous. Again, Theorem \ref{Thm_cmptbl_indcbl} implies that the pair $(1,\psi)$ is inducible. Let  $\tau(\gamma) = (1,\psi)$ for some $\gamma \in \mathrm{Aut}_V(\L)$. Clearly, $\gamma \in \mathrm{Aut}^V(\L)$ with $\tau_2(\gamma) = \psi$, which implies that $\psi \in Im(\tau_2)$. Therefore, $	Im(\tau_2) = Ker(\lambda_2)$. 
\end{proof}

\subsection{Applications to semi-direct product Lie-Yamaguti algebras} 

We use the exact sequences $(A)$ and $(B)$ from Theorem \ref{Thm_two_exact_seq} to describe certain automorphism groups of semi-direct product Lie-Yamaguti algebras. Let $L$ be a Lie-Yamaguti algebra and $(\rho,D,\theta;V)$ be a representation of $L$. Consider the semi-direct product Lie-Yamaugit algebra $L \ltimes V$. Then $V$ is an abelian ideal in $L \ltimes V$,while $L$ is a subalgebra of $L \ltimes V$. Hence, we obtain the following abelian extension of the Lie-Yamaguti algebra $L$
\[\begin{tikzcd}
	0 & V & {L \ltimes V} & L & 0,
	\arrow[from=1-1, to=1-2]
	\arrow["i", from=1-2, to=1-3]
	\arrow["p", from=1-3, to=1-4]
	\arrow["0", from=1-4, to=1-5]
\end{tikzcd}\]
where $i(a)=(0,a)$ and $p(a,v)=a$. The above short exact sequence splits in the category of Lie-Yamaguti algebra with a section $s:L \to L \ltimes V$ defined as $s(a):=(a,0)$. Let us denote
\smallskip
\[\mathrm{Aut}_{\mathrm{rep}}(V):=\text{The group of all automorphism of the representation } (\rho,D,\theta;V).\] 
and
\[\mathrm{Aut}^{\mathrm{inv}}(L):=\{\psi \in \mathrm{Aut}(L)~| ~\text{the representation $(\rho,D,\theta;V)$ is invariant under $\psi$}\}.\]

\noindent In other words, $\mathrm{Aut}^{\mathrm{inv}}(L)$ is the set of those automorphisms of $L$ for which the twisted representation on $V$ is the same as the original one (for the definition of twisted representation, we refer to section \ref{sec-4}). It easily follows that $$\mathrm{Aut}_{\mathrm{rep}}(V) = \C_1 \quad \text{and}\quad\mathrm{Aut}^{\mathrm{inv}}(L)=\C_2.$$
With these notations, we rewrite the sequences $(A)$ and $(B)$ as follows
\[\begin{tikzcd}
	0 & {\mathrm{Aut}^{V,L}(\L)} & {\mathrm{Aut}_V^L(\L)} & {\mathrm{Aut}_{\mathrm{rep}}(V)} & {\mathrm{H}^{(2,3)}(L,V)}, \\
	0 & {\mathrm{Aut}^{V,L}(\L)} & {\mathrm{Aut}^V(\L)} & {\mathrm{Aut}^{\mathrm{inv}}(L)} & {\mathrm{H}^{(2,3)}(L,V)}.
	\arrow[from=1-1, to=1-2]
	\arrow["i_1", from=1-2, to=1-3]
	\arrow["{\tau_1}", from=1-3, to=1-4]
	\arrow["{\lambda_1}", from=1-4, to=1-5]
	\arrow[from=2-1, to=2-2]
	\arrow["i_2", from=2-2, to=2-3]
	\arrow["{\tau_2}", from=2-3, to=2-4]
	\arrow["{\lambda_2}", from=2-4, to=2-5]
\end{tikzcd}\]
In the following theorem, we simplify certain automorphism groups of $L \ltimes V$ in terms of semi-direct product of groups. 

\begin{Thm}
Let $L$ be a Lie-Yamaguti algebra and $(\rho,D,\theta;V)$ be a representation of $L$. Then we have the following isomorphism 
\[\mathrm{Aut}^L_V(L \ltimes V) \cong \mathrm{Aut}_{\mathrm{rep}}(V) \ltimes \mathrm{Aut}^{V,L}(L \ltimes V), \quad\mathrm{Aut}^V(L\ltimes V) \cong \mathrm{Aut}^{\mathrm{inv}}(L) \ltimes \mathrm{Aut}^{V,L}(L \ltimes V).\]
\end{Thm}
\begin{proof}
Recall the following short exact sequence splits in the category of Lie-Yamaguti algebras
\[\begin{tikzcd}
	0 & V & {L \ltimes V} & L & 0.
	\arrow[from=1-1, to=1-2]
	\arrow["i", from=1-2, to=1-3]
	\arrow["p", from=1-3, to=1-4]
	\arrow["0", from=1-4, to=1-5]
\end{tikzcd}\]
Thus, the section $s$ is a Lie-Yamaguti algebra morphism. By equations \eqref{ind_cocycle-I}-\eqref{ind_cocycle-II}, it follows that the maps $\lambda_1:\mathrm{Aut}_L(V) \to \mathrm{H}^{(2,3)}(L,V)$ and $\lambda_2:\mathrm{Aut}_V(L) \to \mathrm{H}^{(2,3)}(L,V)$ in $(A)$ and $(B)$ are trivial. Thus, we have
\[\begin{tikzcd}
	& 0 & {\mathrm{Aut}^{V,L}(L\ltimes V)} & {\mathrm{Aut}_V^L(L\ltimes V)} & {\mathrm{Aut}_{\mathrm{rep}}(V)} & 0, & (I) \\
	& 0 & {\mathrm{Aut}^{V,L}(L\ltimes V)} & {\mathrm{Aut}^V(L\ltimes V)} & {\mathrm{Aut}^{\mathrm{inv}}(L)} & 0. & (II)
	\arrow[from=1-2, to=1-3]
	\arrow["i_1", from=1-3, to=1-4]
	\arrow["{\tau_1}", from=1-4, to=1-5]
	\arrow[from=1-5, to=1-6]
	\arrow[from=2-2, to=2-3]
	\arrow["i_2", from=2-3, to=2-4]
	\arrow["{\tau_2}", from=2-4, to=2-5]
	\arrow[from=2-5, to=2-6]
\end{tikzcd}\]
Next, we show that the above short exact sequences split in the category of groups.

\smallskip

\noindent{\it $(I)$ is split exact:} Let $\eta:\mathrm{Aut}_{\mathrm{rep}}(V) \to \mathrm{Aut}_V^L(L \ltimes V)$ defined as $\eta(\phi):=\gamma$, where $\gamma$ is given by $\gamma(a,v):=(a,\phi(v))$ for all $a \in L$ and $v \in V$. At first we verify $\gamma \in \mathrm{Aut}_V^L(L \ltimes V)$. It is clear that $\gamma(V) = V$ and $\gamma$ induces identity map on $L$. To prove $\gamma$ is a Lie-Yamaguti algebra homomorphism, let $(a,u),(b,v),(c,w) \in L \ltimes V$. Then 
\begin{align*}
	\gamma[(a,u),(b,v)]_{\ltimes} &~=~ \gamma([a,b],\rho(a)(v) - \rho(b)(u)) \\
	&~=~ ([a,b],\phi\rho(a)(v) - \phi\rho(b)(u)) \\
	&~=~ ([a,b],\rho(a)\phi(v) - \rho(b) \phi(u)) \\
	&~=~ [\gamma(a,u),\gamma(b,v)]_{\ltimes},  \\ 
	\gamma\{(a,u),(b,v),(c,w)\}_{\ltimes} &~=~ \gamma(\{a,b,c\}, D(a,b)(w) + \theta(b,c)(u) - \theta(a,c)(v)) \\
	&~=~ (\{a,b,c\}, \phi D(a,b)(w) + \phi \theta(b,c)(u) - \phi\theta(a,c)(v)) \\
	&~=~ (\{a,b,c\}, D(a,b)(\phi(w)) + \theta(b,c)(\phi(u)) - \theta(a,c)(\phi(v))) \\
	&~=~ \{\gamma(a,u),\gamma(b,v),\gamma(c,w)\}_{\ltimes}.
\end{align*} 
So $\gamma$ is a Lie-Yamaguti algebra homomorphism. It is easy to check that $\gamma$ is a bijection. Thus $\gamma \in \mathrm{Aut}_V^L(L \ltimes V)$. Further $\tau_1 \circ \eta = id$, so $\eta$ is a section and $\eta$ is a group homomorphism. Therefore the exact sequence $(I)$ splits and the first isomorphism holds, that is 
\[
\mathrm{Aut}^L_V(L \ltimes V) \cong \mathrm{Aut}_{\mathrm{rep}}(V) \ltimes \mathrm{Aut}^{V,L}(L \ltimes V). 
\] 

\noindent{\it $(II)$ is split exact:} Let $\eta':\mathrm{Aut}^{\mathrm{inv}}(L) \to \mathrm{Aut}^V(L \ltimes V)$ defined by $\eta'(\psi) := \gamma$ where $\gamma$ is defined as $\gamma(a,v) := (\psi(a),v)$. It is clear that $\gamma$ fixes elements of $V$ pointwise and is a bijection. Next we prove that $\gamma$ is a homomorphism. Let $(a,u),(b,v),(c,w) \in L \ltimes V$
\begin{align*}
	\gamma[(a,u),(b,v)]_{\ltimes} &~=~ \gamma([a,b],\rho(a)(v) - \rho(b)(u)) \\
	&~=~ (\psi[a,b],\rho(a)(v) - \rho(b)(u)) \\
	&~=~ ([\psi(a),\psi(b)],\rho(\psi(a))\phi(v) - \rho(\psi(b)) \phi(u)) \\
	&~=~ [\gamma(a,u),\gamma(b,v)]_{\ltimes},  \\ 
	\gamma\{(a,u),(b,v),(c,w)\}_{\ltimes} &~=~ \gamma(\{a,b,c\}, D(a,b)(w) + \theta(b,c)(u) - \theta(a,c)(v)) \\
	&~=~ (\psi\{a,b,c\}, D(a,b)(w) + \theta(b,c)(u) - \theta(a,c)(v)) \\
	&~=~ (\{\psi(a),\psi(b),\psi(c)\}, D(\psi(a),\psi(b))(w) + \theta(\psi(b),\psi(c))(u) - \theta(\psi(a),\psi(c))(v)) \\
	&~=~ \{\gamma(a,u),\gamma(b,v),\gamma(c,w)\}_{\ltimes}.
\end{align*} 
Therefore $\gamma \in \mathrm{Aut}^V(L \ltimes V)$. Also $\tau_2 \circ \eta' = id$, so $\eta'$ is a section. It is also easy to verify that $\eta'$ is a group homomorphism. Hence $(II)$ splits in the category of groups, giving us the second isomorphism, that is 
\[\mathrm{Aut}^V(L\ltimes V) \cong \mathrm{Aut}^{\mathrm{inv}}(L) \ltimes \mathrm{Aut}^{V,L}(L \ltimes V).\]
Hence completing the proof of the theorem. 
\end{proof}

\bigskip

\section{\large The particular case of nilpotent Lie-Yamaguti algebras}\label{sec-6}

\medskip

In this section, we consider nilpotent Lie-Yamaguti algebras of index $2$ and discuss the inducibility problem for an abelian extension arising from these algebras. Finally, we give an algorithm to find all inducible pairs of automorphisms for an abelian extension arising from a nilpotent Lie-Yamaguti algebra of index 2. 

\smallskip 

We first recall the definition of nilpotent Lie-Yamaguti algebra from \cite{abdelwahab}. Let $L$  be a Lie-Yamaguti algebra. Denote
\begin{equation*}
L^{(0)}:=L \quad\mbox{and}\quad	L^{(n+1)} := \left[L^{(n)},L\right] + \left\{L^{(n)},L,L\right\} + \left\{L,L,L^{(n)}\right\}, \quad 
\mbox{for}~  n\geq 0.
\end{equation*}
Then, a Lie-Yamaguti algebra $L$ is \textbf{nilpotent} of \textbf{index $k$} if $k$ is the least positive integer such that $L^{(k)} = 0$.

\smallskip 

Let $L$ be a nilpotent Lie-Yamaguti algebra of index $2$. Note that the center $\Z(L)$ is non-trivial and the induced Lie-Yamaguti algebra structure on the quotient space $L/\Z(L)$ is abelian. We consider the following extension arising from $L$
\[\begin{tikzcd}
	0 & \Z(L) & L & L/\Z(L) & 0,
	\arrow[from=1-1, to=1-2]
	\arrow["i", from=1-2, to=1-3]
	\arrow["p", from=1-3, to=1-4]
	\arrow[from=1-4, to=1-5]
\end{tikzcd}\]
where $i$ is the inclusion map and $p$ is the projection map. We denote the quotient $ L/\Z(L)$ by $\widebar{L}$. The induced representation of $\widebar{L}$ on $\Z(L)$ are given by the representation maps $\rho:\widebar{L} \to End(\Z(L))$ and $D,\theta:\widebar{L} \times \widebar{L} \to End(\Z(L))$, which are defined as follows: for all $a,b \in \widebar{L}\mbox{ and } v \in \Z(L),$ 
\begin{align*}
	\rho(a)(v) &:= [s(a),v]_{L} = 0,  \\
	D(a,b)(v) &:= \{s(a), s(b),v\}_{L} = 0, \\
	\theta(a,b)(v) &:= \{v, s(a), s(b)\}_{L} = 0, \quad (\text{since}~ v \in \Z(L)).
\end{align*}
Thus, the induced representation of $\widebar{L}$ on $\Z(L)$ is trivial. As a consequence, every pair $(\phi,\psi) \in \mathrm{Aut}(\Z(L)) \times \mathrm{Aut}(\widebar{L})$ is a compatible pair. Now since $\widebar{L}$ is an abelian Lie-Yamaguti algebra, we have, 
\[B^{(2,3)}(L,V) = \{0\},\]  
due to which Theorem $\ref{Thm_cmptbl_indcbl}$ reduces to the following lemma.
  
\begin{Lem} \label{Lem_counter-exm}
Let $L$ be a nilpotent Lie Yamaguti algebra of index $2$ with center $\Z(L)$. Consider the abelian extension
\[\begin{tikzcd}
	0 & \Z(L) & L & \widebar{L} & 0.
	\arrow[from=1-1, to=1-2]
	\arrow["i", from=1-2, to=1-3]
	\arrow["p", from=1-3, to=1-4]
	\arrow[from=1-4, to=1-5]
\end{tikzcd}\] 
A pair $(\phi,\psi) \in \mathrm{Aut}(\Z(L)) \times \mathrm{Aut}(\widebar{L})$ is inducible if and only if for all $a,b,c \in \widebar{L}$
\[
	\phi \circ \alpha (a,b) = \alpha \circ (\psi,\psi) (a,b)  \quad \text{and} \quad 
	\phi \circ \beta (a,b,c) = \beta \circ (\psi,\psi,\psi) (a,b,c).
\]
\end{Lem}

\medskip

We now introduce two families of nilpotent Lie-Yamaguti algebras of index $2$ with one dimensional center. 
 
\begin{Exm}\label{Counter_exm_1}
	For $n \ge 1$, let $\mathfrak{H}_n$ be the vector space over $\mathbb{K}$ generated by $2n+1$ elements $$\{e,e_1,\ldots,e_n,e_{n+1},\ldots,e_{2n}\}.$$ The Lie-Yamaguti brackets on $\mathfrak{H}_n$ are defined as follows 
	\[
	[e_i,e_{n+i}] = e,\quad \{e_i,e_{n+i},e_i\} = e, \quad 1 \le i \le n.
	\]   
	All other brackets on the basis elements are either zero or are determined by the relations of Lie-Yamaguti algebra. 
\end{Exm} 

\begin{Rem}
Note that the pair $(\mathfrak{H}_n,[\cdot,\cdot])$ is the $n$-dimensional Heisenberg Lie algebra. Therefore, we call this family \textbf{``Heisenberg Lie-Yamaguti algebras"}. Also note, $(\mathfrak{H}_n,\{\cdot,\cdot,\cdot\})$ is a Lie triple system.  In general, a Lie-Yamaguti algebra does not simultaneously possess a non-trivial Lie algebra structure and a non-trivial Lie triple system structure. However, Example \ref{Counter_exm_1} yields a special type of Lie-Yamaguti algebra that is simultaneously a Lie algebra and a Lie triple system, with both the brackets being non-trivial.
\end{Rem}

\smallskip\noindent
Now consider the following abelian extension
\[\begin{tikzcd}  
	& & & 0 & \Z(\mathfrak{H}_n) & \mathfrak{H}_n & \widebar{\mathfrak{H}}_n & 0, &
	\arrow[from=1-4, to=1-5]
	\arrow["i", from=1-5, to=1-6]
	\arrow["p", from=1-6, to=1-7]
	\arrow[from=1-7, to=1-8]
\end{tikzcd}\]  
where $i$ and $p$ are the inclusion and the projection maps, respectively. Since $\mathfrak{H}_n$ is a nilpotent Lie-Yamaguti algebra of index $2$, Lemma \ref{Lem_counter-exm} characterizes the inducibility problem for the above extension. Note that $\mathcal{Z}(\mathfrak{H}_n) = \operatorname{span}\{e\}$. We further observe the following points regarding the above extension.

\medskip

\begin{itemize}
	\item Any pair $(\phi,\psi) \in \mathrm{Aut}(\Z(\mathfrak{H}_n)) \times \mathrm{Aut}(\widebar{\mathfrak{H}}_n)$ is a compatible pair.  
	\smallskip
	\item $\widebar{\mathfrak{H}}_n = \mathrm{span}\{\bar{e}_1,\bar{e}_2, \ldots, \bar{e}_n,\bar{e}_{n+1},\ldots,\bar{e}_{2n}\}$ where $\widebar{e}_i = p(e_i)$.
	\smallskip
	\item For the rest of our discussion, we fix a section $s: \widebar{\mathfrak{H}}_n \to \mathfrak{H}_n$ of $p$ given by $s(\bar{e}_i)=e_i$. 
	\smallskip
	\item Since $\Z(\mathfrak{H}_n) = \mathrm{span}\{e\}$, any automorphism $\phi \in \mathrm{Aut}(\Z(\mathfrak{H}_n))$ is determined by some element $\kappa \in \K$, $\kappa \neq 0$.
	\smallskip
	\item Let $\psi \in \mathrm{Aut}(\widebar{\mathfrak{H}}_n)$. The matrix of $\psi$ with respect to the basis $\{\bar{e}_1,\bar{e}_2,\ldots,\bar{e}_n,\bar{e}_{n+1},\ldots,\bar{e}_{2n}\}$ can be written as  
	\[
	[\psi]
	=\begin{bmatrix}
			A_{n \times n}
			& \vline & 
			B_{n \times n}\\
		 	\hline 
			C_{n \times n} 
			& \vline &
		D_{n \times n}
	\end{bmatrix}.
	\]
\end{itemize} 

\medskip

\noindent
Here we introduce a notation that will be used to state the main theorem of this section. Denote $\M_{(ac,bd)}$ by 
\[\M_{(ac,bd)} =
\begin{bmatrix}
\begin{vmatrix}
a_{1,1} & c_{1,1} \\
b_{1,\sigma} & d_{1,\sigma}
\end{vmatrix}
& 
\begin{vmatrix}
a_{1,2} & c_{1,2} \\
b_{1,\sigma} & d_{1,\sigma}
\end{vmatrix}
&\cdots&
\begin{vmatrix}
a_{1,n} & c_{1,n} \\
b_{1,\sigma} & d_{1,\sigma}
\end{vmatrix} 
\\
\begin{vmatrix}
a_{2,1} & c_{2,1} \\
b_{2,\sigma} & d_{2,\sigma}
\end{vmatrix} 
& 
\begin{vmatrix}
a_{2,2} & c_{2,2} \\
b_{2,\sigma} & d_{2,\sigma}
\end{vmatrix}
& \cdots & 
\begin{vmatrix}
a_{2,n} & c_{2,n} \\
b_{2,\sigma} & d_{2,\sigma}
\end{vmatrix}
\\ \vdots & \vdots & & \vdots \\
\begin{vmatrix}
a_{n,1} & c_{n,1} \\
b_{n,\sigma} & d_{n,\sigma}
\end{vmatrix}
& 
\begin{vmatrix}
a_{n,2} & c_{n,2} \\
b_{n,\sigma} & d_{n,\sigma}
\end{vmatrix}
& \cdots & 
\begin{vmatrix}
a_{n,n} & c_{n,n} \\
b_{n,\sigma} & d_{n,\sigma}
\end{vmatrix}
\end{bmatrix},
\quad\quad 1 \le \sigma \le n.
\]

\bigskip

\noindent
With the above discussion in mind, we state the main result of this section. 

\begin{Thm} \label{Theorem-1_counter_exm}
A pair $(\phi,\psi) \in \mathrm{Aut}(\Z(\mathfrak{H}_n)) \times \mathrm{Aut}(\widebar{\mathfrak{H}}_n)$ is inducible if and only if the following conditions hold.
\begin{enumerate}
\item $A^{t}D - C^{t}B = \kappa I_{n,n}$.
\smallskip
\item $A^{t}C$ and $B^{t}D$ are symmetric matrices.  
\smallskip
\item $A^t\M_{(ac,bd)} = \kappa E_{\sigma, \sigma}$, for all $1 \le \sigma \le n$. Here, $E_{\sigma,\sigma}$  denote the elementary matrix that has $1$ in the $(\sigma,\sigma)^{th}$ entry as its only non-zero entry. 
\smallskip
\item For all $1 \le \sigma \le n$,
\begin{enumerate}
\item $A^t\M_{(ac,ac)} = O_{n \times n} = B^t\M_{(ac,ac)}$,
\smallskip
\item $A^t\M_{(bd,bd)} = O_{n \times n} = B^t\M_{(bd,bd)},$
\smallskip	
\item $C^t \M_{(ac,bd)} = O_{n \times n}$.
\end{enumerate}
\end{enumerate}
\end{Thm}

\begin{proof}
Let $(\phi,\psi) \in \mathrm{Aut}(\Z(\mathfrak{H}_n)) \times \mathrm{Aut}(\widebar{\mathfrak{H}}_n)$ be an inducible pair. Using Lemma \ref{Lem_counter-exm}, we have
\begin{equation} \label{Exm_equ_1}
\phi \circ \alpha (a,b) = \alpha \circ (\psi,\psi) (a,b),
\end{equation}
\begin{equation} \label{Exm_equ_2}
\phi \circ \beta (a,b,c) = \beta \circ (\psi,\psi,\psi) (a,b,c),\quad \forall ~a,b,c \in \widebar{\mathfrak{H}}_n.
\end{equation}
We show that equation (\ref{Exm_equ_1}) gives us conditions $(1)$ and $(2)$, and equation (\ref{Exm_equ_2}) gives us condition $(3)$ and $(4)$. First, we observe that for any $1 \le i,j,k \le 2n$,
\[\alpha(\bar{e}_i,\bar{e}_j) = [s(\bar{e}_i),s(\bar{e}_j)] - s[\bar{e}_i,\bar{e}_j] = [s(\bar{e}_i),s(\bar{e}_j)] = [e_i,e_j],\] 
\[\beta(\bar{e}_i,\bar{e}_j,\bar{e}_k) = \{s(\bar{e}_i),s(\bar{e}_j),s(\bar{e}_k)\} - s\{\bar{e}_i,\bar{e}_j,\bar{e}_k\} = \{e_i,e_j,e_k\}.\]
Therefore, equations (\ref{Exm_equ_1}) and (\ref{Exm_equ_2}) reduces to
\begin{equation} \label{Exm_equ_3}
\alpha(\psi(\bar{e}_i),\psi(\bar{e}_j)) = \phi([e_i,e_j]),  \qquad 1 \le i,j \le 2n,
\end{equation}
\begin{equation} \label{Exm_equ_4}
\beta(\psi(\bar{e}_i),\psi(\bar{e}_j),\psi(\bar{e}_k)) = \phi(\{e_i,e_j,e_k\}), \qquad 1 \le i,j,k \le 2n.
\end{equation}
Recall that we denote the matrix of $\psi$ with respect to the basis $\{\bar{e}_1,\cdots,\bar{e}_n,\bar{e}_{n+1},\cdots,\bar{e}_{2n}\}$ by 
\[[\psi]=\begin{bmatrix}
A_{n \times n}
& \vline & 
B_{n \times n}\\
\hline
C_{n \times n} 
& \vline &
D_{n \times n}
\end{bmatrix}.\]
For $1 \le \iota \le n$, 
\begin{equation} \label{Exm_equ_5}
\psi(\bar{e}_{\iota}) = \sum_{r=1}^{n} a_{r,\iota} \bar{e}_r + \sum_{r=1}^{n} c_{r,\iota} \bar{e}_{n+r}, 
\end{equation}
and for $n+1 \le \iota \le 2n$, 
\begin{equation} \label{Exm_equ_6}
\psi(\bar{e}_{\iota}) = \sum_{r=1}^{n} b_{r,\iota-n} \bar{e}_r + \sum_{r=1}^{n} d_{r,\iota-n} \bar{e}_{n+r}. 
\end{equation}

\medskip 
\noindent
\textbf{Proving condition $(1)$:} $A^{t}D - C^{t}B = \kappa I_{n \times n}$. 

\smallskip

Let $1 \le i \le n$ and $n+1 \le j \le 2n$. Then plugging equation (\ref{Exm_equ_5}) and (\ref{Exm_equ_6}) in the left hand side of  (\ref{Exm_equ_3}), we get  
\begin{align*}
\alpha(\psi(\bar{e}_i), \psi(\bar{e}_j)) 
&= \left[s\left(\sum_{r=1}^{n} a_{r,i} \bar{e}_r + \sum_{r=1}^{n} c_{r,i} \bar{e}_{n+r}\right), s\left(\sum_{r=1}^{n} b_{r,j-n} \bar{e}_r + \sum_{r=1}^{n} d_{r,j-n} \bar{e}_{n+r}\right)\right]\\
&= \left[\sum_{r=1}^{n} a_{r,i} e_r + \sum_{r=1}^{n} c_{r,i} e_{n+r},~ \sum_{r=1}^{n} b_{r,j-n} e_r + \sum_{r=1}^{n} d_{r,j-n} e_{n+r} \right] \\
&= \sum_{r=1}^{n}\left(a_{r,i}d_{r,j-n} - c_{r,i}b_{r,j-n}\right) [e_r,e_{n+r}]\\
&= \sum_{r=1}^{n}\left(a_{r,i}d_{r,j-n} - c_{r,i}b_{r,j-n}\right) e.
\end{align*}
Then from (\ref{Exm_equ_3}) it follows that, for all $1 \le i \le n$ and $n+1 \le j \le 2n$, 
\[\sum_{r=1}^{n}\left(a_{r,i}d_{r,j-n} - c_{r,i}b_{r,j-n}\right) 
=
\begin{cases}
\kappa, & \text{if}~ j-n = i, \\
0, & \text{otherwise}.
\end{cases}\]
or equivalently taking $j-n=l$, we rewrite the above equation, for $1 \le i,l \le n$,
\[\sum_{r=1}^{n}\left(a_{r,i}d_{r,l} - c_{r,i}b_{r,l}\right) 
=
\begin{cases}
\kappa, & \text{if}~ l = i, \\
0, & ~\text{otherwise}.
\end{cases}\]
which is equivalent to condition $(1)$, i.e.,
\begin{equation} \label{Exm_equ_7}
	A^tD - C^tB = \kappa I_{n \times n}.
\end{equation}
	
\medskip
\noindent
\textbf{Proving condition $(2)$:} $A^tC$ and $B^tD$ are symmetric matrices. 

\smallskip

Let $1 \le i,j \le n$. Plugging  equation (\ref{Exm_equ_5}) in the left hand side of (\ref{Exm_equ_3}), we get 
\[\alpha(\psi(\bar{e}_i), \psi(\bar{e}_j)) =\sum_{r=1}^{n} \left(a_{r,i}c_{r,j} - c_{r,i}a_{r,j}\right) e.\]
Again, from (\ref{Exm_equ_3}) it follows that, for all $1 \le i,j \le n$,
\[\sum_{r=1}^{n} \left(a_{r,i}c_{r,j} - c_{r,i}a_{r,j}\right) = \phi([e_i,e_j]) = 0,\]
which is equivalent to condition $(2)$, 
\begin{equation} \label{Exm_equ_8}
A^tC - C^tA = 0,~\text{i.e.},~A^tC ~\text{is symmetric.} 
\end{equation}
Similarly, computing equation (\ref{Exm_equ_3}) for $n+1 \le i,j \le 2n$, we get the condition
\begin{equation} \label{Exm_equ_9}
B^tD - D^tB = 0,~\text{i.e.}, ~B^tD ~\text{is symmetric.}
\end{equation} 

\medskip
\noindent
\textbf{Proving condition $(3)$:} $A^t\M_{(ac,bd)} = \kappa E_{\sigma, \sigma}$, \quad for all $1 \le \sigma \le n$.

\smallskip
	
Let $1 \le i, k \le n$ and $n+1 \le j \le 2n$. Then, plugging equations (\ref{Exm_equ_5}) and (\ref{Exm_equ_6}) in the left hand side of (\ref{Exm_equ_4}), we get 
\begin{align*}
&\beta(\psi(\bar{e}_i), \psi(\bar{e}_j), \psi(\bar{e}_k)) \\
&= \left\{s\left(\sum_{r=1}^{n} a_{r,i} \bar{e}_r + \sum_{r=1}^{n} c_{r,i} \bar{e}_{n+r}\right), s\left(\sum_{r=1}^{n} b_{r,j-n} \bar{e}_r + \sum_{r=1}^{n} d_{r,j-n} \bar{e}_{n+r}\right), s\left(\sum_{r=1}^{n} a_{r,k} \bar{e}_r + \sum_{r=1}^{n} c_{r,k} \bar{e}_{n+r}\right)\right\}  \\
&= \left\{\sum_{r=1}^{n} a_{r,i} e_r + \sum_{r=1}^{n} c_{r,i} e_{n+r},~ \sum_{r=1}^{n} b_{r,j-n} e_r + \sum_{r=1}^{n} d_{r,j-n} e_{n+r},~\sum_{r=1}^{n} a_{r,k} e_r + \sum_{r=1}^{n} c_{r,k} e_{n+r} \right\} \\
&= \sum_{r=1}^{n} (a_{r,i} d_{r,j-n} a_{r,k} - c_{r,i}b_{r,j-n}a_{r,k}) \{e_r,e_{n+r},e_{r}\}\\
&= \sum_{r=1}^{n} a_{r,k}\left(a_{r,i} d_{r,j-n} - c_{r,i}b_{r,j-n}\right) e.
\end{align*}
Then, from equation \eqref{Exm_equ_4} it follows that, for all $1 \le i,k \le n$ and $n+1 \le j \le 2n$, 
\[\sum_{r=1}^{n} a_{r,k}\left(a_{r,i} d_{r,j-n} - c_{r,i}b_{r,j-n}\right)=
\begin{cases}
\kappa, & \text{if}~ i=j-n=k, \\
0, & \text{otherwise}.
\end{cases}\]
or equivalently taking $j-n = \sigma$ we have, for all $1 \le i,\sigma,k \le n$,
\[\sum_{r=1}^{n} a_{r,k}\left(a_{r,i} d_{r,\sigma} - c_{r,i}b_{r,\sigma}\right)=
\begin{cases}
\kappa, & \text{if}~ i=\sigma=k, \\
0, & \text{otherwise}.
\end{cases}\]
which is equivalent to condition $(3)$, i.e., 
\begin{equation} \label{Exm_equ_10}
A^t\M_{(ac,bd)} = \kappa E_{\sigma, \sigma}, \quad \text{for all}~ 1 \le \sigma \le n.
\end{equation}
	
\medskip
\noindent
\textbf{Proving condition $(4)$:}  

\smallskip
	
Let $1 \le i,j,k \le n$. Using the equation (\ref{Exm_equ_5}) in the left hand side of the equation (\ref{Exm_equ_4}), we get
\[\beta(\psi(\bar{e}_i), \psi(\bar{e}_j), \psi(\bar{e}_k)) = \sum_{r=1}^{n} \left(a_{r,i}c_{r,j}a_{r,k} - c_{r,i} a_{r,j} a_{r,k}\right)e.\]
From equation (\ref{Exm_equ_4}) it follows that
\[\sum_{r=1}^{n} a_{r,k} (a_{r,i} c_{r,j} - c_{r,i}a_{r,j}) = 0,\quad \forall~1 \le i,j,k \le n.\]
Replacing $j = \sigma$, the above condition is equivalent to the following expression
\begin{equation} \label{Exm_equ_11}
A^t \M_{(ac,ac)} = O_{n,n},\quad \forall~1 \le \sigma \le n.
\end{equation}
Similarly, taking $1 \le i,j \le n$ and $n+1 \le k \le 2n,$ we get 
\[\beta(\psi(\bar{e}_i), \psi(\bar{e}_j), \psi(\bar{e}_k)) = \sum_{r=1}^{n} (a_{r,i}c_{r,j}b_{r,k-n} - c_{r,i} a_{r,j} b_{r,k-n} )e.\]
Then from equation (\ref{Exm_equ_4}) it follows that 
\[\sum_{r=1}^{n} b_{r,k-n} \left(a_{r,i} c_{r,j} - c_{r,i}a_{r,j}\right) = 0,\quad 1 \le i, j \le n \mbox{ and }n+1 \le k \le 2n.\]
Replacing $k-n=l$ and $i = \sigma$, the above condition is equivalent to the following identity 
\begin{equation} \label{Exm_equ_12}
B^t \M_{(ac,ac)} = O_{n \times n},\quad \forall~1 \le \sigma \le n. 
\end{equation}
From equations (\ref{Exm_equ_11}) and (\ref{Exm_equ_12}) we get
\begin{equation} \label{Exm_equ_13}		
A^t\M_{(ac,ac)} = O_{n \times n} = B^t\M_{(ac,ac)},\quad \forall~1 \le \sigma \le n. 
\end{equation}
Similarly, computing (\ref{Exm_equ_4}) with $n+1 \le i,j \le 2n$, $1 \le k \le n$ we get condition $4(b)$. The condition $4(c)$ is obtained by computing (\ref{Exm_equ_4}) with $1 \le i \le n$, $n+1 \le j,k \le 2n$.   
	
\medskip
\noindent
\textbf{Converse:} It easily follows that the conditions $(1)$-$(4)$ are equivalent to (\ref{Exm_equ_3}) and (\ref{Exm_equ_4}), which in turn are equivalent to (\ref{Exm_equ_1}) and (\ref{Exm_equ_2}). Thus, any pair $(\phi,\psi)$ satisfying conditions $(1)$-$(4)$ satisfies (\ref{Exm_equ_1}) and (\ref{Exm_equ_2}), hence is an inducible pair. Therefore, the conditions $(1)$-$(4)$ stated in the theorem are necessary and sufficient for a pair of automorphisms $(\phi,\psi) \in \mathrm{Aut}(\Z(\mathfrak{H}_n)) \times \mathrm{Aut}(\widebar{\mathfrak{H}}_n)$ to be inducible by an automorphism $\gamma \in \mathrm{Aut}(\mathfrak{H}_n)$.  
\end{proof}

\medskip 

\begin{Rem}
It is clear that not all compatible pairs are inducible. For example, in the above theorem a compatible pair $(\phi,\psi) \in \mathrm{Aut}(\Z(\mathfrak{H}_n)) \times \mathrm{Aut}(\widebar{\mathfrak{H}}_n)$ is not inducible unless the conditions $(1)$-$(4)$ are satisfied. 
\end{Rem}
\medskip

\noindent \textbf{Application to the case of Heisenberg Lie algebras.} Let $(\mathfrak{H}_n,[\cdot,\cdot])$ be the $n$-dimensional Heisenberg Lie algebra. $(\mathfrak{H}_n,[\cdot,\cdot])$ becomes a Lie-Yamaguti algebra with trivial ternary bracket. 
Consider the following abelian extension in the category of Lie-Yamaguti algebras
\[\begin{tikzcd}  
	& & & 0 & \Z(\mathfrak{H}_n) & \mathfrak{H}_n & \widebar{\mathfrak{H}}_n & 0, & &
	\arrow[from=1-4, to=1-5]
	\arrow["i", from=1-5, to=1-6]
	\arrow["p", from=1-6, to=1-7]
	\arrow[from=1-7, to=1-8]
\end{tikzcd}\]  
where $i$ and $p$ are the inclusion and the projection maps, respectively. Then, taking the ternary bracket to be zero in the proof of  Theorem \ref{Theorem-1_counter_exm}, we obtain the following result. 
\begin{Cor}
 A pair of Lie algebra automorphisms $(\phi,\psi) \in \mathrm{Aut}(\Z(\mathfrak{H}_n)) \times \mathrm{Aut}(\widebar{\mathfrak{H}}_n)$ is inducible by a Lie algebra automorphism $\gamma \in \mathrm{Aut}(\mathfrak{H}_n)$ if and only if 
\smallskip
\begin{enumerate}
	\item $A^{t}D - C^{t}B = \kappa I_{n,n}$.
	\smallskip
	\item $A^{t}C$ and $B^{t}D$ are symmetric matrices.
\end{enumerate}
\end{Cor}
\noindent

\smallskip

Next, we consider another family of nilpotent Lie Yamaguti algebra of index $2$, which contains Heisenberg Lie-Yamaguti algebra as a sub-algebra.  

\begin{Exm} \label{Counter_exm_2}
For $n \ge 1$, let $\mathfrak{G}_n$ denote the vector space generated by $2n+2$ elements $$\{e,e_1,\ldots,e_n,e_{n+1},e_{n+2},\ldots,e_{2n},e_{2n+1}\}.$$ 
\noindent
The Lie-Yamaguti brackets on $\mathfrak{G}_n$ are defined as follows 
\[[e_i,e_{n+1+i}] = e = \{e_i,e_{n+1+i},e_i\}~~ \text{for}~ 1 \le i \le n, \quad \text{and} \quad \{e_{n+1},e_{2n+1},e_{n+1}\} = e.\]      
All other brackets on the basis elements are either zero or are determined by the relations of Lie-Yamaguti algebra.
\end{Exm}

\begin{Rem} 
The family of Heisenberg Lie-Yamaguti algebras can be embedded in this family. More precisely, for each $n \ge 1$, there exist an injective Lie-Yamaguti algebra morphism $\Phi:\mathfrak{H}_n \to \mathfrak{G}_n$, defined as follows
\[\Phi(e) := e, \quad \text{and} \quad \Phi(e_i):= 
\begin{cases}
e_i & \text{if}~~ 1 \le i \le n, \\
e_{i+1} & \text{if}~  n+1 \le i \le 2n. 
\end{cases}\]
\smallskip
\noindent
Therefore, we call this family \textbf{``generalized Heisenberg Lie-Yamaguti algebra"}.
\end{Rem}

\noindent
Now, consider the following abelian extension
\[\begin{tikzcd}  
	0 & \Z(\mathfrak{G}_n) & \mathfrak{G}_n & \widebar{\mathfrak{G}}_n & 0,
	\arrow[from=1-1, to=1-2]
	\arrow["i", from=1-2, to=1-3]
	\arrow["p", from=1-3, to=1-4]
	\arrow[from=1-4, to=1-5]
\end{tikzcd}\] 
where $i$ is the inclusion map and $p$ is the projection map. Let us note down the following observations.

\begin{itemize}
\item $\Z(\mathfrak{G}_n) = \mathrm{span}\{e\}$, that is, the center of $\mathfrak{G}_n$ is one-dimensional. 
\smallskip
\item Any pair $(\phi,\psi) \in \mathrm{Aut}(\Z(\mathfrak{G}_n)) \times \mathrm{Aut}(\widebar{\mathfrak{G}}_n)$ is a compatible pair.  
\smallskip
\item $\widebar{\mathfrak{G}}_n = \mbox{span}\{\bar{e}_1, \ldots, \bar{e}_n,\bar{e}_{n+1},\bar{e}_{n+2},\ldots,\bar{e}_{2n+1}\},$ where $\widebar{e}_i = p(e_i)$.
\smallskip
\item For our further disucssion, we fix a section $s: \widebar{\mathfrak{G}}_n \to \mathfrak{G}_n$ of $p$ that is given by $s(\bar{e}_i)=e_i$. 
\smallskip
\item Since $\Z(\mathfrak{G}_n) = \mbox{span}\{e\}$, any automorphism $\phi \in \mathrm{Aut}(\Z(\mathfrak{G}_n))$ is determined by some element $\kappa \in \K$, $\kappa \neq 0$.
\smallskip
\item For any automorphism, $\psi \in \mathrm{Aut}(\widebar{\mathfrak{G}}_n)$, we denote the matrix of $\psi$ with respect to the basis 
$\{\bar{e}_1, \ldots, \bar{e}_n,\bar{e}_{n+1},\bar{e}_{n+2},\ldots,\bar{e}_{2n+1}\}$ by 

\bigskip 

\begin{center}
\begin{tikzpicture}
	\node[] at (-1,0) (Psi) {$ [\psi] $};
	\node[] at (-0.4,0) (eql) {$ = $};
	\node[] at (0.8,0) () {$m_{n+1,1}$};
	\node[] at (2,0) () {$m_{n+1,2}$};
	\node[] at (3,0) () {$\cdots$};
	\node[] at (4.1,1.6) () {$m_{1,n+1}$};
	\node[] at (4.1,1.0) () {$m_{2,n+1}$};
	\node[] at (4.1,0.6) () {$\vdots$};
	\node[] at (4.1,0) () {$m_{n+1,n+1}$};
	\node[] at (4.1,-0.6) () {$m_{n+2,n+1}$};
	\node[] at (4.1,-1.0) () {$\vdots$};
	\node[] at (4.1,-1.6) () {$m_{2n+1,n+1}$};
	\node[] at (5.7,0) () {$m_{n+1,n+2}$};
	\node[] at (6.9,0) () {$\cdots$};
	\node[] at (8.1,0) () {$m_{n+1,2n+1}$};
	\node[] at (1.7,1.3) () {\LARGE{$A_{n \times n}$}};
	\node[] at (1.7,-1.2) () {\LARGE{$C_{n \times n}$}};
	\node[] at (7,1.3) () {\LARGE{$B_{n \times n}$}};
	\node[] at (7,-1.2) () {\LARGE{$D_{n \times n}$}};
	\draw[black,thick,-] (0,2)--(0,-2);
	\draw[black,thick,-] (0,2)--(0.1,2);
	\draw[black,thick,-] (0,-2)--(0.1,-2);
	\draw[black,thick,-] (0.1,0.4)--(3.2,0.4);
	\draw[black,thick,-] (3.2,0.4)--(3.2,1.9);
	\draw[black,thick,-] (0.1,-0.3)--(3.2,-0.3);
	\draw[black,thick,-] (3.2,-0.3)--(3.2,-1.9);
	\draw[black,thick,-] (4.9,0.4)--(4.9,1.9);
	\draw[black,thick,-] (4.9,0.4)--(9.0,0.4);
	\draw[black,thick,-] (4.9,-0.3)--(9.0,-0.3);
	\draw[black,thick,-] (4.9,-0.3)--(4.9,-1.9);
	\draw[black,thick,-] (9.1,2)--(9.1,-2);
	\draw[black,thick,-] (9.1,2)--(9,2);
	\draw[black,thick,-] (9.1,-2)--(9,-2);
\end{tikzpicture}
\end{center}
\end{itemize} 

\medskip

\noindent
It is clear that Lemma \ref{Lem_counter-exm} is applicable for the above extension. Thus, we obtain a explicit necessary and sufficient condition for inducibility of a pair of automorphism $(\phi,\psi) \in \mathrm{Aut}(\mathfrak{G}_n) \times \mathrm{Aut}(\widebar{\mathfrak{G}}_n)$, by solving the following equations
\begin{eqnarray*} 
	\phi \circ \alpha (\bar{e}_i,\bar{e}_j) &=& \alpha \circ (\psi,\psi) (\bar{e_i},\bar{e}_j), \\
	\phi \circ \beta (\bar{e}_i,\bar{e}_j,\bar{e}_k) &=& \beta \circ (\psi,\psi,\psi) (\bar{e}_i,\bar{e}_j,\bar{e}_k),
\end{eqnarray*} 
for all $1 \le i,j,k \le 2n+1$. Computing the above equations for each value of $i,j,$ and $k$, one gets $24$ relations characterizing the inducibility problem for the extension
\[\begin{tikzcd}  
	0 & \Z(\mathfrak{G}_n) & \mathfrak{G}_n & \widebar{\mathfrak{G}}_n & 0 
	\arrow[from=1-1, to=1-2]
	\arrow["i", from=1-2, to=1-3]
	\arrow["p", from=1-3, to=1-4]
	\arrow[from=1-4, to=1-5]
\end{tikzcd}.\] 
For instance, taking $i=n+1=k$, $j=2n+1$, we get
\[
\sum_{r=1}^n m_{r,n+1} 
\begin{vmatrix}
	m_{r,n+1} & m_{n+1+r,n+1} \\
	b_{r,n} & d_{r,n} 
\end{vmatrix}
+ 
m_{n+1,n+1}
\begin{vmatrix}
	m_{n+1,n+1} & m_{n+1,2n+1} \\
	m_{2n+1,n+1} & d_{n,n} 
\end{vmatrix} = \kappa,\]
which gives a necessary condition for a pair $(\phi, \psi) \in \mathrm{Aut}(\Z(\mathfrak{G}_n)) \times \mathrm{Aut}(\widebar{\mathfrak{G}}_n)$ to be inducible. 

\smallskip

\subsection*{An algorithm for nilpotent Lie-Yamaguti algebras} 
Now, we write an algorithm to find the inducible pairs for the following extension that arises from nilpotent Lie-Yamaguti algebras of index 2. 
\[\begin{tikzcd}  
	0 & \Z(L) & L & \widebar{L} & 0, 
	\arrow[from=1-1, to=1-2]
	\arrow["i", from=1-2, to=1-3]
	\arrow["p", from=1-3, to=1-4]
	\arrow[from=1-4, to=1-5]
\end{tikzcd}\]
where $\Z(L)$ is the center of $L$ and $\widebar{L}=L/ \Z(L)$.

\bigskip
\noindent
\textbf{Notations:}
\smallskip
\begin{enumerate}
\item For any $m,n\geq 1$, let $L = \mathrm{span}\{e_1,\ldots, e_n,e_{n+1},e_{n+2},\ldots,e_{n+m}\}$ be an $n+m$-dimensional nilpotent Lie-Yamaguti algebra of index $2$ with $m$-dimensional center $\Z(L)$ such that $\{e_{n+1},\ldots, e_{n+m}\}$ is a basis of $\Z(L)$.

\smallskip

\item Let $B_{L} := \{e_{1},e_{2},\ldots,e_{n}\}$, and by $S_{L}:=\{\text{Set of all structure equations of}~L\}$. For instance, the set of all structure equations of Heisenberg Lie-Yamaguti algebra is, $S_{\mathfrak{H}_n} =\big\{[e_i,e_{n+i}] = e, ~\{e_i,e_{n+i},e_i\} = e ~:~ 1 \le i \le n\big\}$.

\smallskip

\item Consider the abelian extension $0 \to \Z(L) \xrightarrow[]{i} L \xrightarrow[]{p} \widebar{L} \to 0$. Fix a basis of $\widebar{L}= \mathrm{span}\{\bar{e}_1,\bar{e}_2,\ldots,\bar{e}_n\}$, along with a section $s:L \to \L$ of $p$ defined as  $s(\bar{e}_i) = e_{i}$. Let $B_{\widebar{L}} := \{\bar{e}_1,\bar{e}_2,\ldots,\bar{e}_n\}$.

\smallskip

\item For any pair $(\phi,\psi) \in \mathrm{Aut}(\Z(L)) \times \mathrm{Aut}(\widebar{L})$, let $[\psi] = (a_{ij})_{n \times n}$ be the matrix representation of $\psi$ and $[\phi] = (b_{ij})_{m \times m}$ be the $m \times m$ matrix associated to $\phi$. 
\end{enumerate}

\smallskip
\noindent
With the above notations, Lemma \ref{Theorem-1_counter_exm} reduces to the following statement: A pair $(\phi,\psi) \in \mathrm{Aut}(\Z(L)) \times \mathrm{Aut}(\widebar{L})$ is inducible if and only if for all $\bar{e}_i,\bar{e}_j,\bar{e}_k \in B_{\widebar{L}}$
\begin{eqnarray*}
	\alpha(\psi(\bar{e}_i),\psi(\bar{e}_j)) & = & \phi([e_{i},e_{j}]), \\
	\beta(\psi(\bar{e}_i),\psi(\bar{e}_j),\psi(\bar{e}_k)) & = & \phi(\{e_{i},e_{j},e_{k}\}).  
\end{eqnarray*}

\medskip
\noindent
We are now ready to write our algorithm.
 
\begin{algorithm}
	\centering
	\caption{: \textbf{For nilpotent Lie-Yamaguti algebras of index $2$}}
	\begin{algorithmic}[1]
		\Statex \textbf{Input:} 
		$B_{L},~B_{\widebar{L}},~S_{L},~s,~\psi,~\phi,~\alpha,~\beta$
		\smallskip
		\hrule
		\hfil
		\State $\mathcal{R} \leftarrow \emptyset $  
		\State For $ i=1 ~to~ n   $
		\State \quad For $ j=1 ~to~ n $ 
		\State \qquad Simplify $ \alpha(\psi(\bar{e}_i),\psi(\bar{e}_j)) =  \phi([e_{i},e_{j}])$ and get a relation r.
		\State  \qquad $ \mathcal{R} \leftarrow  \mathcal{R} \cup r $
		\State \qquad For $ k=1 ~to ~ n $
		\State \quad \qquad Simplify $\beta(\psi(\bar{e}_i),\psi(\bar{e}_j),\psi(\bar{e}_k)) = \phi(\{e_{i},e_{j},e_{k}\})$ and get a relation r.
		\State  \quad \qquad $ \mathcal{R} \leftarrow  \mathcal{R} \cup r $
		\State Return $ \mathcal{R} $
	\end{algorithmic}
\end{algorithm}
\noindent
The set $\mathcal{R}$ obtained from the algorithm as the output contains all the relations that characterize the inducibility problem. 

\section*{\large{Conclusion}}  

In this work, we discuss the inducibility problem for abelian extensions of Lie-Yamaguti algebras. We present an algorithm to characterize the inducible pairs of automorphisms in abelian extensions arising from nilpotent Lie-Yamaguti algebras (of index $2$) with a one-dimensional center. For example, one can deduce a solution of the inducible problem for the $3$-dimensional nilpotent Lie-Yamaguti algebras $\mathcal{L}_{3,1},~\mathcal{L}_{3,2},~\mathcal{L}_{3,3},$ and $\mathcal{L}_{3,4}$, described in \cite{abdelwahab}. Our first example of Heisenberg Lie-Yamaguti algebras is a generalization of $\mathcal{L}_{3,4}$. 

\medskip

In literature, different kinds of extensions such as central, abelian, and non-abelian extensions of algebraic structures are developed in \cite{hochschild,hoch,serre,lyndon,eilen}. The non-abelian extensions are the most general notion. Eilenberg and Maclane first studied non-abelian extensions of abstract groups in \cite{eilen} and introduced low-dimensional non-abelian cohomology groups. Similar results for (super) Lie algebras, Leibniz algebras, Lie $2$-algebras, $L_{\infty}$-algebras, and Rota-Baxter Lie algebras are also considered in the literature, we refer to \cite{casas,Chen,fial-pen,hazra-habib,hoch,inas,laza,liu,Mishra}. Thus, it is imperative to study non-abelian extensions for Lie-Yamaguti algebras and to interpret them in terms of derivations and non-abelian cohomology.

\end{document}